\documentclass[onefignum,onetabnum]{siamart190516}

\usepackage{color}
\usepackage{mathdots}
\usepackage{amsmath,amssymb,amsfonts}
\usepackage{textcomp}
\usepackage{wrapfig}
\usepackage{fancybox}
\usepackage{pgf,tikz}
\usetikzlibrary{arrows}
\usepackage{booktabs}
\usepackage{hyperref}
\usepackage{csquotes}

\newcommand{\RR}{\mathbb{R}}
\newcommand{\NN}{\mathbb{N}}

\usepackage{lipsum}
\usepackage{amsfonts}
\usepackage{graphicx,epstopdf} 
\usepackage[caption=false]{subfig} 
\usepackage{algorithmic}
\ifpdf
\DeclareGraphicsExtensions{.eps,.pdf,.png,.jpg}
\else
\DeclareGraphicsExtensions{.eps}
\fi

\newsiamremark{remark}{Remark}
\newsiamremark{hypothesis}{Hypothesis}
\crefname{hypothesis}{Hypothesis}{Hypotheses}
\newsiamthm{claim}{Claim}

\headers{EPH curves}{L. Romani and A. Viscardi}

\title{Construction and evaluation of PH curves in exponential-polynomial spaces
	\thanks{\funding{This work was partially funded by INdAM-GNCS 2020 project \enquote{Interpolation and smoothing: theoretical, computational and applied aspects} (Prot. U-UFMBAZ-2020-000564).}}
}

\author{
	Lucia Romani\thanks{$AM^2$ - Dipartimento di Matematica, Alma Mater Studiorum Università di Bologna, Bologna, Italy (\email{lucia.romani@unibo.it})}
	\and Alberto Viscardi\thanks{Dipartimento di Matematica \enquote{Giuseppe Peano}, Università di Torino, Torino, Italy (\email{alberto.viscardi@unito.it})}
}

\usepackage{amsopn}

\ifpdf
\hypersetup{
	pdftitle={Construction and evaluation of PH curves in exponential-polynomial spaces},
	pdfauthor={L. Romani, A. Viscardi}
}
\fi

\begin{document}

\maketitle

\begin{abstract}
	In the past few decades polynomial curves with Pythagorean Hodograph (for short PH curves) have received considerable attention due to their usefulness in various CAD/CAM areas, manufacturing, numerical control machining and robotics.
	This work deals with classes of PH curves built-upon exponential-polynomial spaces (for short EPH curves). In particular, for the two most frequently encountered exponential-polynomial spaces, we first provide necessary and sufficient conditions to be satisfied by the control polygon of the B\'{e}zier-like curve in order to fulfill the PH property. Then, for such EPH curves, fundamental characteristics like parametric speed or arc length are discussed to show the interesting analogies with their well-known polynomial counterparts. Differences and advantages with respect to ordinary PH curves become commendable when discussing the solutions to application problems like the interpolation of first-order Hermite data.
	Finally, a new evaluation algorithm for EPH curves is proposed and shown to compare favorably with the celebrated de Casteljau-like algorithm and two recently proposed methods: Wo\'zny and Chudy's algorithm and the dynamic evaluation procedure by Yang and Hong.
\end{abstract}



\section{Introduction}
Ordinary polynomial curve segments with the Pythagorean-Hodograph (PH) property have been extensively studied \cite{Fbook}, and their construction has been satisfactorily extended also to spaces spanned by algebraic-trigonometric polynomials \cite{CS20,GAPL18,KKRV15,RM19,RSA14}. Although spaces spanned by algebraic-hyperbolic polynomials have close analogies with the ones spanned by algebraic-trigonometric polynomials (see \cref{sec2}), on the one hand they offer complementary solutions and, on the other hand, their handling might require some additional caution which is important to underline. 
Indeed, in the remainder of this manuscript we first show (see \cref{sec3,sec4}) that the constraints to be satisfied by the control points of the algebraic-hyperbolic B\'{e}zier curve segments in order to achieve the PH property, mimick very closely the necessary and sufficient conditions known in the polynomial and algebraic-trigonometric cases. In addition, also the computed expressions for their fundamental characteristics (parametric speed or arc length) sound to be very similar.
{But, when used in application contexts like interpolating $C^1$ Hermite data (see \cref{sec:Hermite}), algebraic-hyperbolic B\'{e}zier curves allow one to get regular curves without undesired loops or self-intersections, whose shapes differ from those achievable by means of algebraic-trigonometric  B\'{e}zier curves}. For instance, when considering the planar Hermite data of \cref{fig:PH}, none of the { four solutions (see \cite[Chapter 25]{Fbook})} provided by the ordinary polynomial PH quintics are free of loops.
{Instead, when the same Hermite problem is solved by using either algebraic-trigonometric PH (for short ATPH) curves or algebraic-hyperbolic PH (for short EPH) curves, for suitable choices of the free parameter (which both families are equipped with) several good solutions exist (see \cref{fig:ATPH_EPH}). The further advantage offered by EPH curves is shown in \cref{fig:reproduction}: when the $C^1$ Hermite data are sampled from some hyperbolic functions, then the EPH Hermite interpolant is able to reconstruct such functions exactly (similar to what ATPH curves do in the trigonometric case). These are of course practical reasons that motivate the study of EPH curves.\\
}
An additional reason that prompted us to investigate algebraic-hyperbolic PH curves arises from the observation that, even if the hyperbolic cosine and sine are just the opposite side of the exponential coin from the trigonometric cosine and sine, the normalized B-basis (also known as Chebyshevian Bernstein basis) of the underlying Extended Chebyshev (EC) space is known to be affected by numerical instability when large exponential shape parameters are selected \cite{Roth}. Thus, one of the main goals of this work is also to suggest a stable formulation of the normalized B-basis of the two exponential-polynomial spaces (or, more precisely, algebraic-hyperbolic spaces) that are most frequently encountered when working with non-polynomial PH curves, so that numerical instabilities are avoided. Furthermore, for such spaces, we aim at proposing a novel evaluation algorithm that is stable for a wide range of the exponential shape parameter, in contrast to the dynamic evaluation procedure in \cite{YH19}, and has a lower computational time (see \cref{sec:eval}), compared with the de Casteljau-like B-algorithm \cite{CP94,MP99,MP07,MP10} (analogue of the de Casteljau algorithm for classical polynomial B\'{e}zier curves), and with the algorithm introduced by Wo\'zny and Chudy in \cite{WC20}.  

\begin{figure}[h!]
	\begin{center}
		\hspace{-0.75cm}	\includegraphics[width=0.75\textwidth]{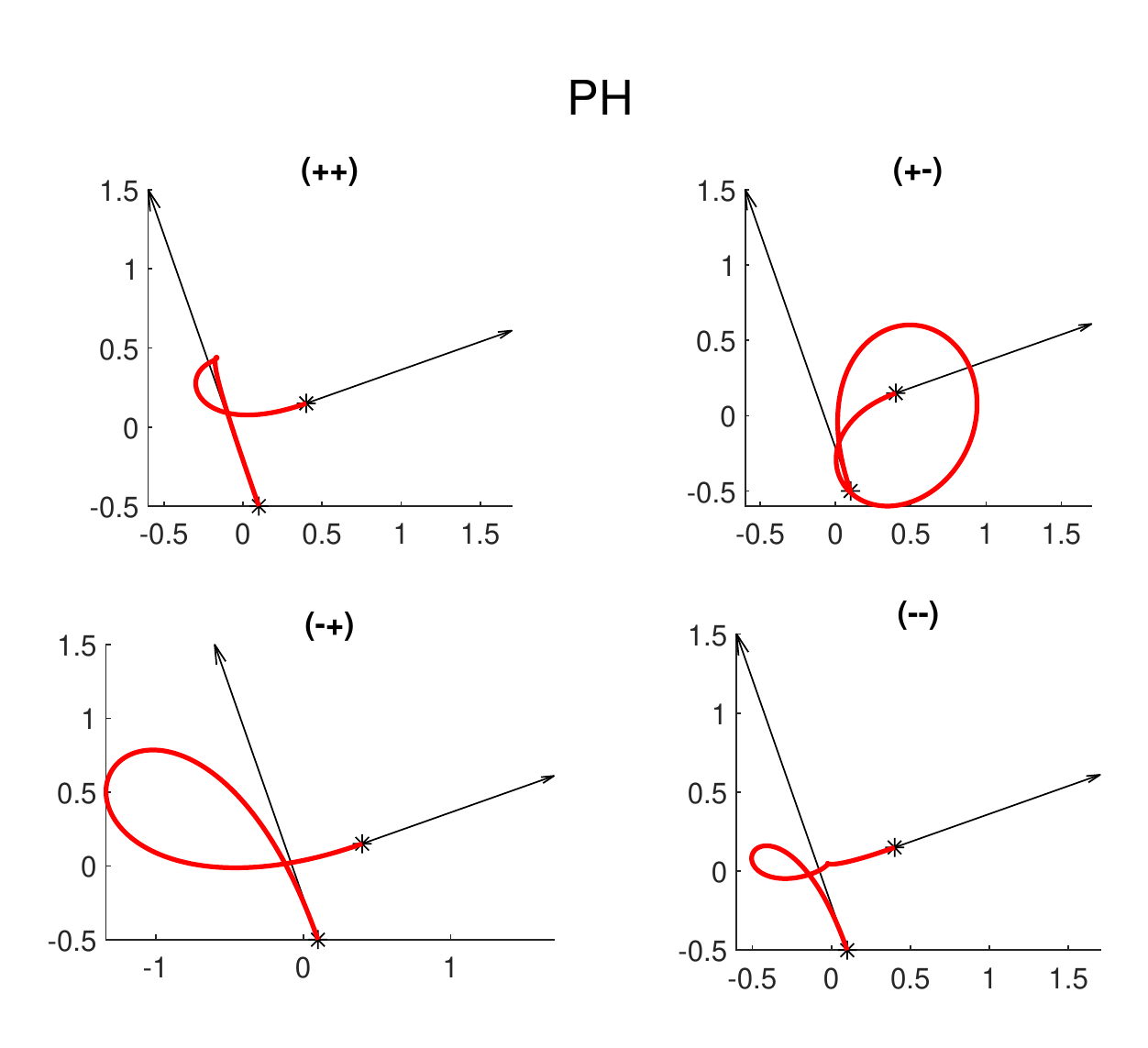}\\
	\end{center}
	\caption{{The four planar PH interpolants to the points $\mathbf{r}_0=(0.1,-0.5)$, $\mathbf{r}_5=(0.4,0.15)$ and associated first derivatives $\mathbf{d}_i=(-3.5,10)$, $\mathbf{d}_f=(6.5,2.3)$ (here plotted with a scale factor of $1/5$ to fit into the picture). As for the meaning of the notation $++$, $+-$, $-+$, $--$ the reader can consult \cite[Chapter 25]{Fbook}.}}
	\label{fig:PH}
\end{figure}

\begin{figure}[h!]
	\begin{center}
		$ $\\
		\hspace{-0.75cm}\includegraphics[width=0.4\textwidth]{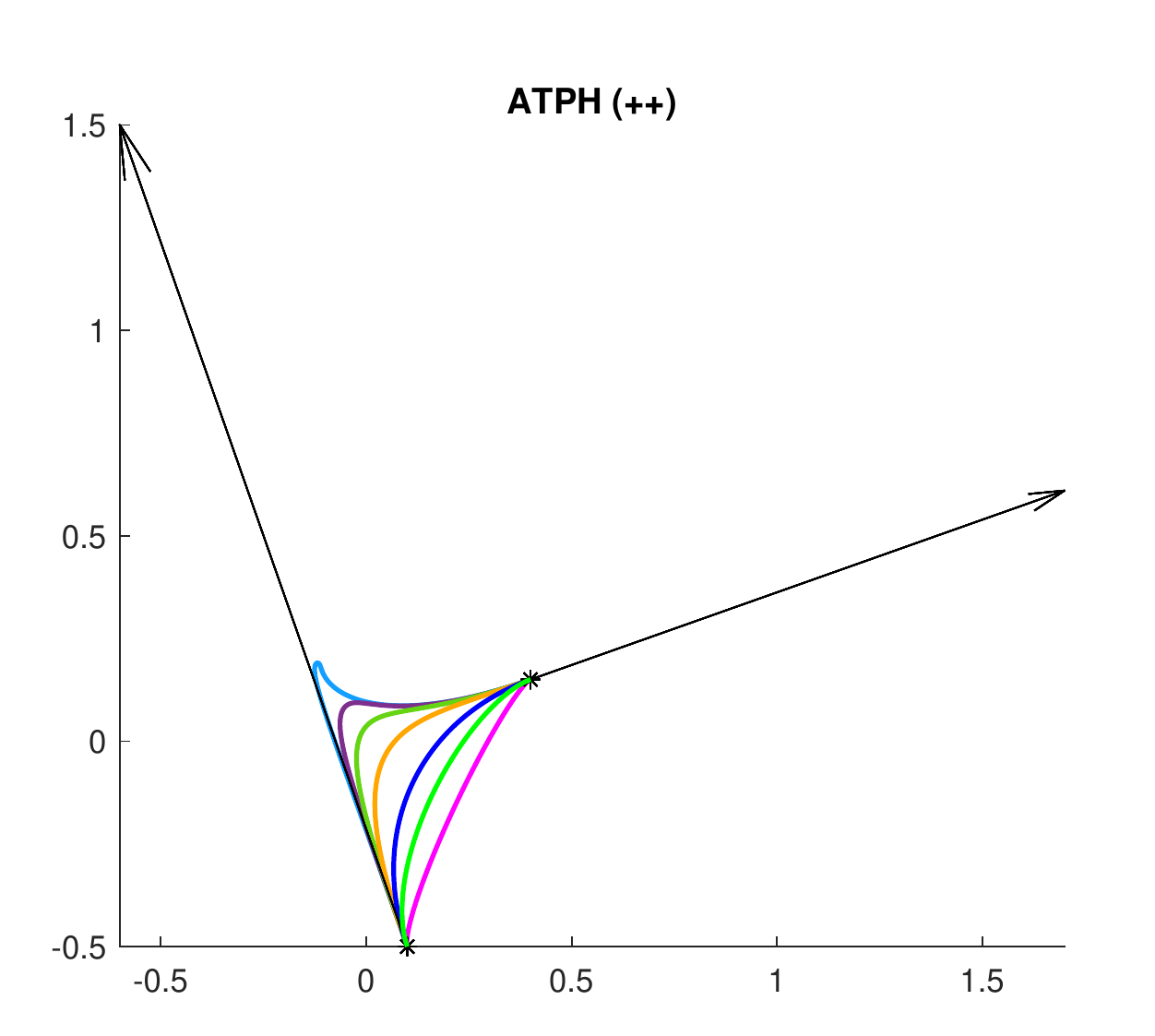}
		\includegraphics[width=0.4\textwidth]{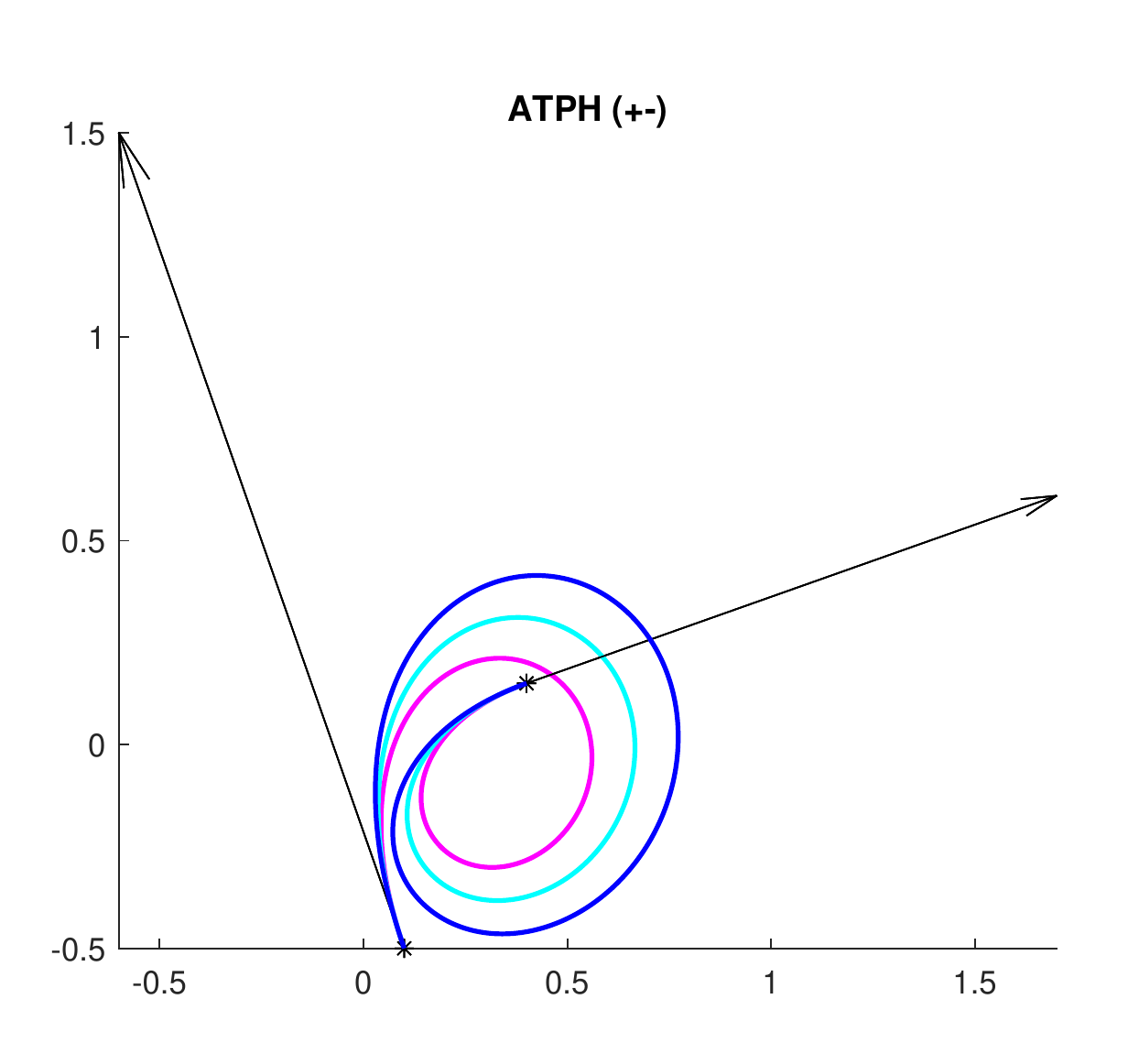}	
		\\
		$ $\\ 
		\hspace{-0.75cm}%
		\includegraphics[width=0.4\textwidth]{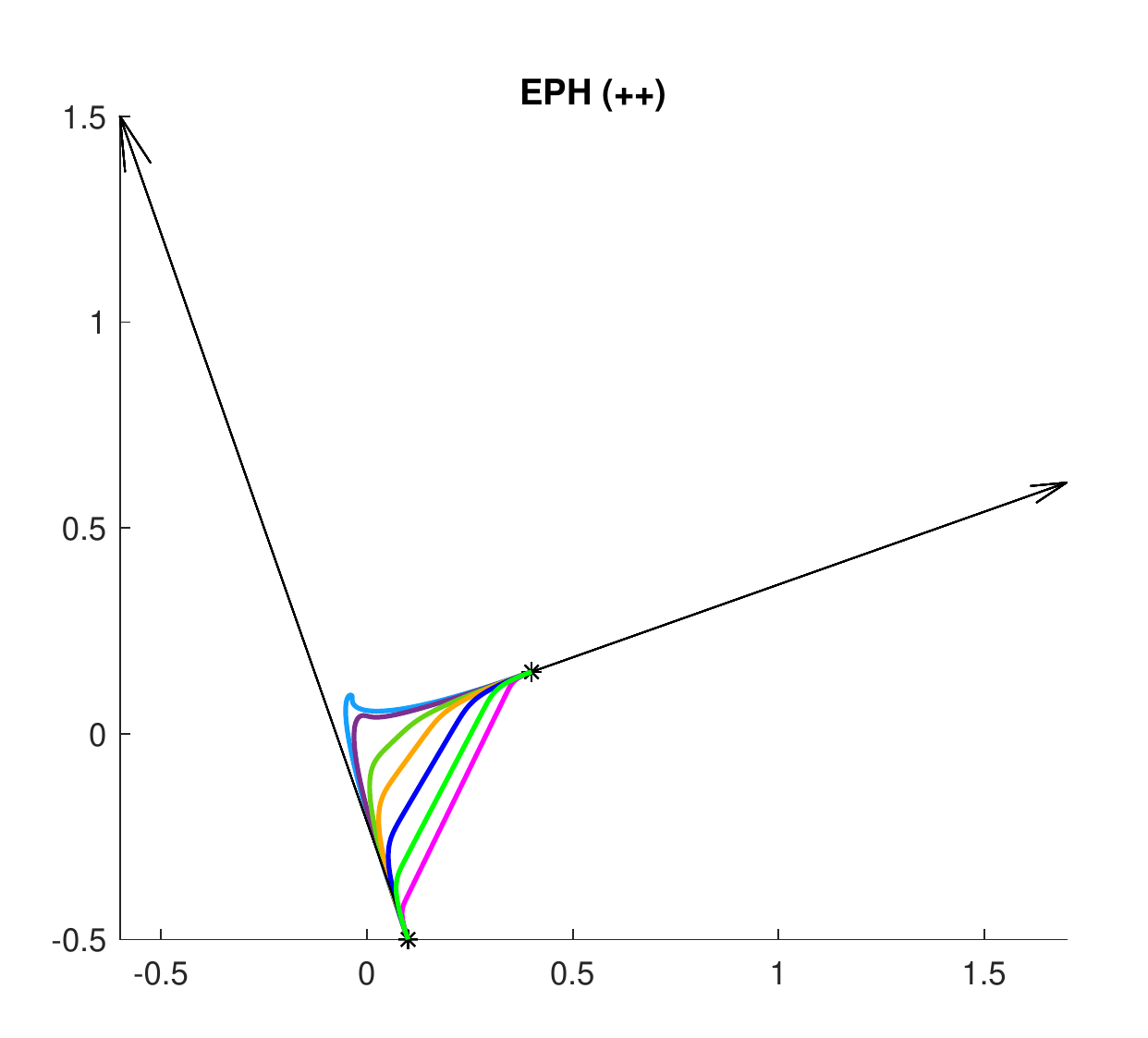}
		\includegraphics[width=0.4\textwidth]{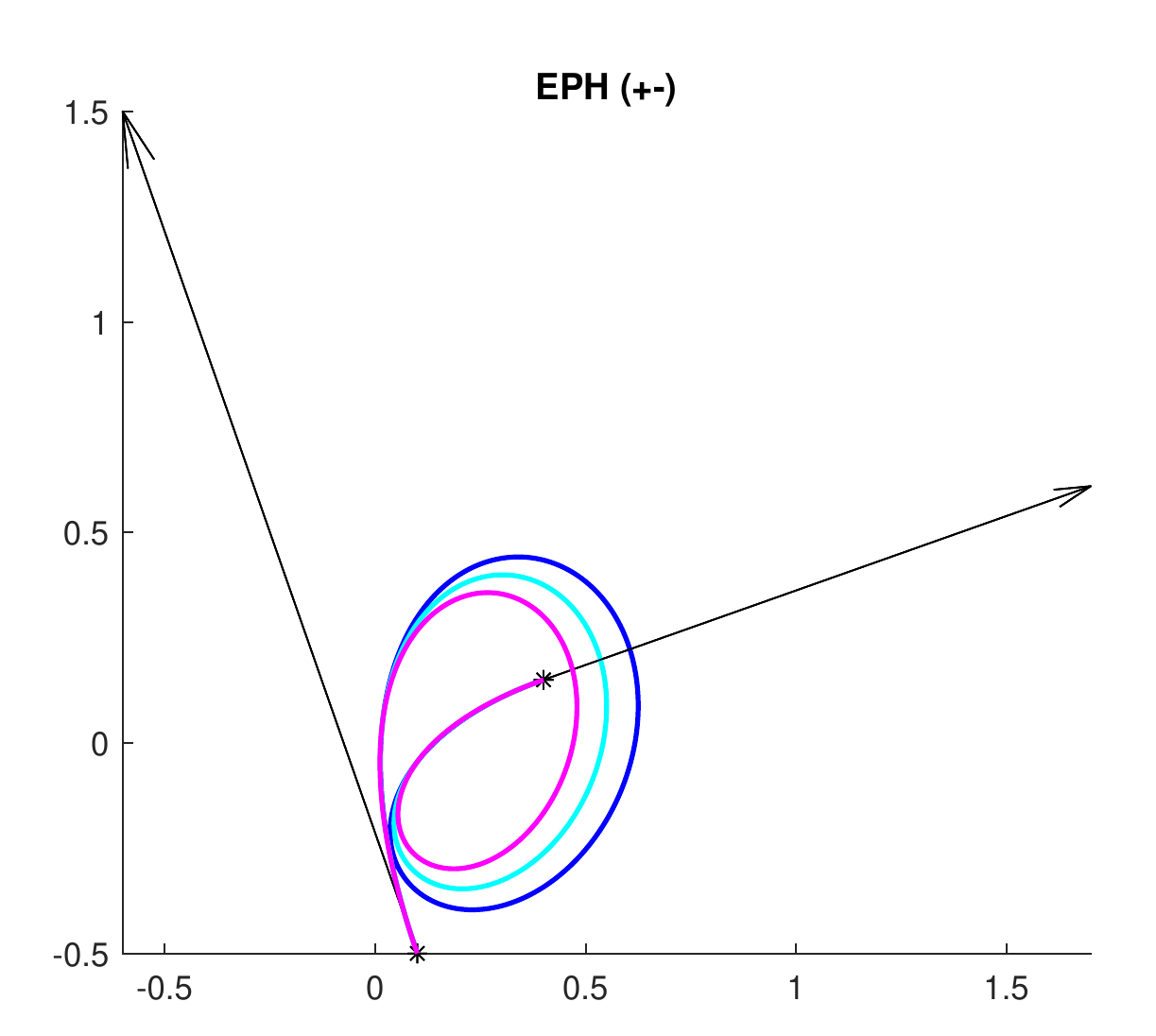}		
	\end{center}
	\caption{{
			Top:	 Planar ATPH interpolants (see \cite{RSA14}) to the same Hermite data of Fig. \ref{fig:PH}. Precisely, on the left the solution $++$ obtained with shape parameter $\alpha \in \{0.01,0.05,0.1,0.2,0.3,0.4,0.6\}$ and on the right the solution $+-$ obtained with shape parameter $\alpha \in \{0.6,0.7,0.8\}$.     
			Bottom: Planar EPH interpolants to the same Hermite data of Fig. \ref{fig:PH}. Precisely, on the left the solution $++$ obtained with the exponential shape parameter $\omega \in \{8,10,15,20,30,50,100\}$ and on the right the solution $+-$ obtained with the exponential shape parameter $\omega \in \{3,3.5,4\}$. As for the meaning of the notation $++$, $+-$ the reader can consult \cite{RSA14} and Section \ref{sec:Hermite}, respectively.}
	}
	\label{fig:ATPH_EPH}
\end{figure}

\begin{figure}[h!]
	\begin{center}
		$ $\\
		\hspace{-0.25cm}
		\includegraphics[width=0.40\textwidth]{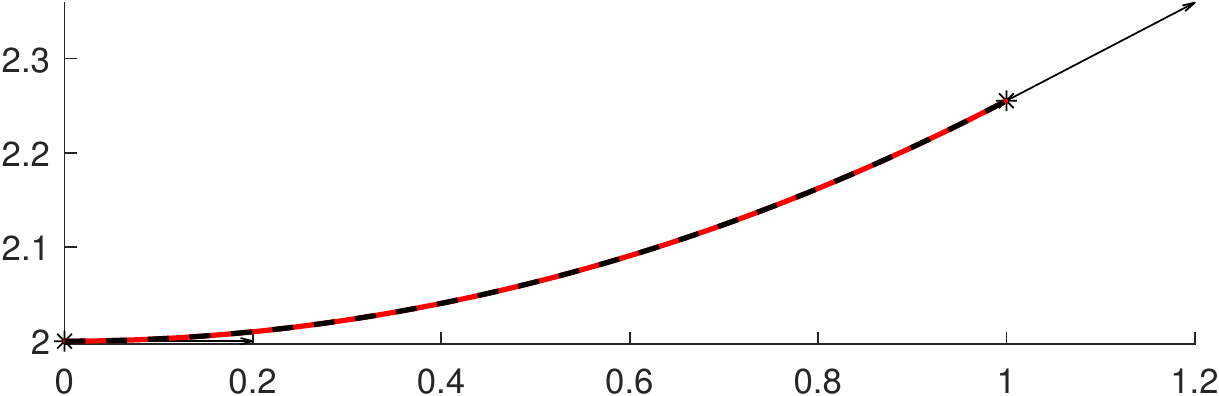}	\hspace{-0.01cm} 
		\includegraphics[width=0.40\textwidth]{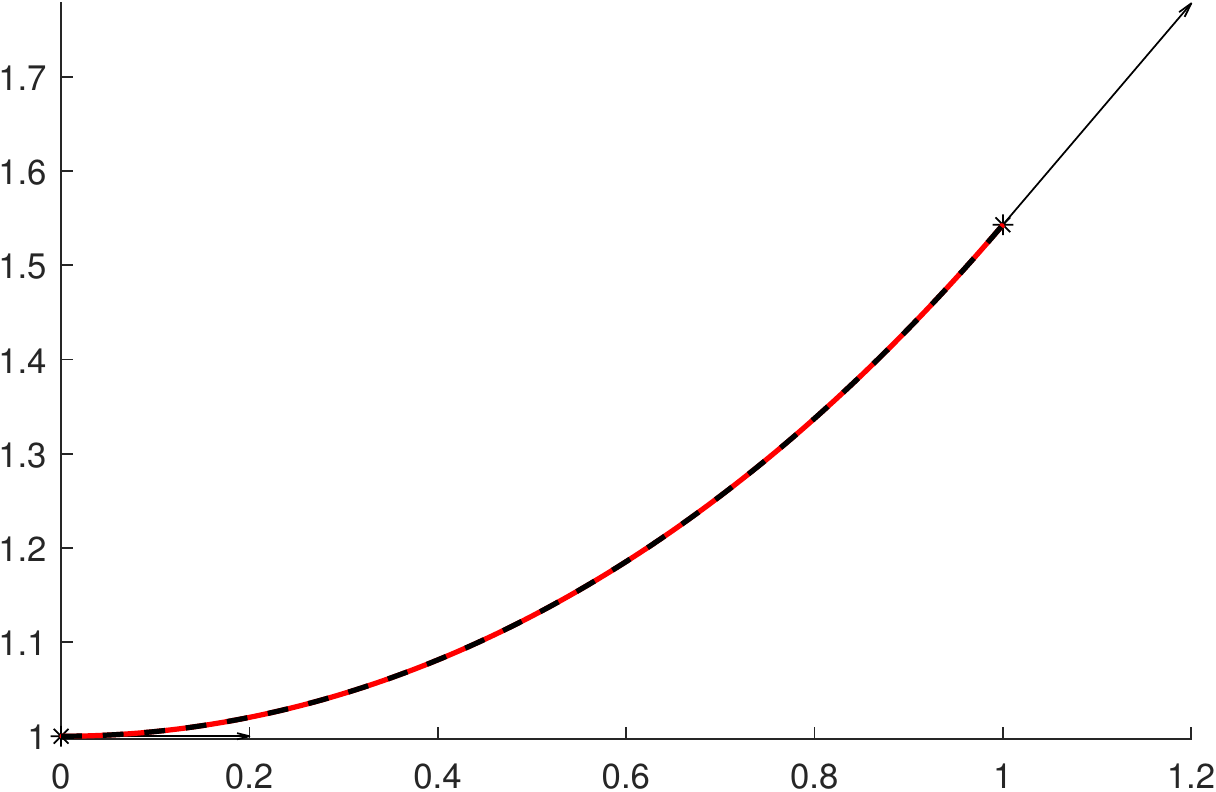}	\hspace{-0.01cm}
		\includegraphics[width=0.18\textwidth]{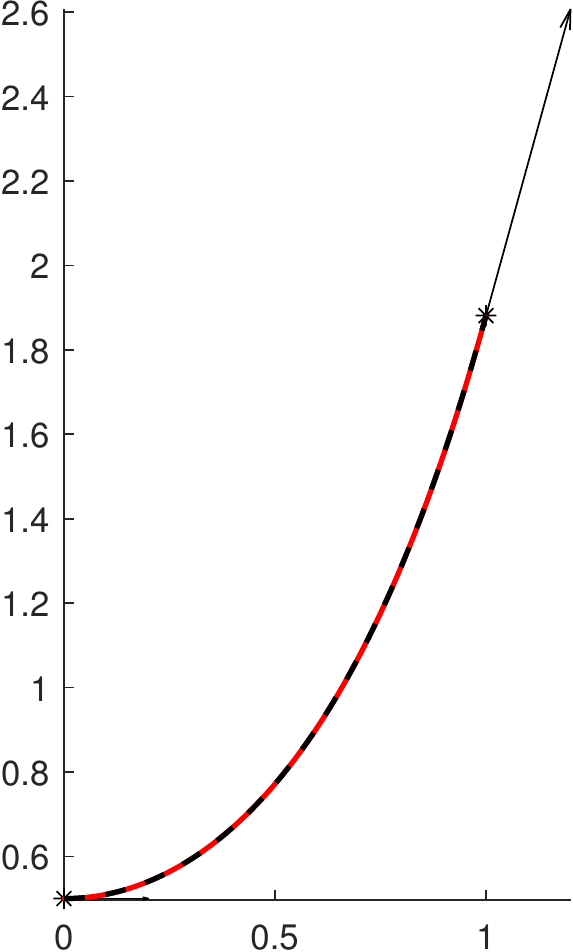}	
	\end{center}
	\caption{{
			Planar EPH Hermite interpolant  (solution $++$) to the points $\mathbf{r}_0=(0,(2 \omega)^{-1})$, $\mathbf{r}_5=(1,(2\omega)^{-1} \, \cosh(2\omega))$ and associated first derivatives $\mathbf{d}_i=(1,0)$, $\mathbf{d}_f=(1, \sinh(2 \omega))$ (here plotted with a scale factor of $1/5$ to fit into the picture), overlapping the function $(2\omega)^{-1} \, \cosh(2\omega t)$, $t \in [0,1]$, for the following choices of the exponential shape parameter: $\omega =0.25$ (left), $\omega =0.5$ (center), $\omega =1$ (right).}
	}
	\label{fig:reproduction}
\end{figure}


\section{PH curves in exponential-polynomial spaces: EPH curves}\label{sec2}

Let $m \in \NN_0$ and $\omega \in \RR^+$, where $\mathbb{N}_0=\mathbb{N}\cup\{0\}$ and $\RR^+$ denotes the set of positive real numbers.
\begin{definition}[Exponential-polynomial spaces] \label{def:EP}
	We define, in terms of $m$ and $\omega$, the following spaces of exponential polynomials
	\[
	{\cal EP}^{\omega}_{m}\;:=\;span\left\{\;\{ 1,t\}\;\cup\;\bigcup_{k=1}^m \{e^{k \omega t}, e^{-k \omega t} \} \; \right\},
	\]
	and
	\[
	{\cal OEP}^{\omega}_{m}\;:=\;span\left\{\;\{ 1,t\}\;\cup\;\bigcup_{k=1}^m \{e^{(2k-1) \omega t}, e^{-(2k-1) \omega t} \} \; \right\}.
	\]
\end{definition}
We observe that
\[
{\cal OEP}^\omega_0\;=\;{\cal EP}^\omega_0\;=\;span\left\{1,t\right\}, \qquad {\cal OEP}^\omega_1\;=\;{\cal EP}^\omega_1\;=\;span\left\{1,t,e^{\omega t},e^{-\omega t}\right\},
\]
but, for $m>1$, ${\cal OEP}^\omega_m\subsetneq{\cal EP}^\omega_{2m-1}$. Moreover, denoting with $\mathcal{D}$ the differential operator $d/dt$, it is easy to check that
\[
{\cal DEP}^{\omega}_{m}\;=\;span\left\{\;\{ 1\}\;\cup\;\bigcup_{k=1}^m \{e^{k \omega t}, e^{-k \omega t} \} \; \right\}\;\subset\;{\cal EP}^{\omega}_{m},
\]
\[
{\cal D}^2{\cal EP}^{\omega}_{m}\;=\;span\left\{\;\bigcup_{k=1}^m \{e^{k \omega t}, e^{-k \omega t} \} \; \right\}\;\subset\;{\cal DEP}^{\omega}_{m},
\]
and similarly
\[
{\cal DOEP}^{\omega}_{m}\;=\;span\left\{\;\{ 1\}\;\cup\;\bigcup_{k=1}^m \{e^{(2k-1) \omega t}, e^{-(2k-1) \omega t} \} \; \right\}\;\subset\;{\cal OEP}^{\omega}_{m},
\]
\[
{\cal D}^2{\cal OEP}^{\omega}_{m}\;=\;span\left\{\;\bigcup_{k=1}^m \{e^{(2k-1) \omega t}, e^{-(2k-1) \omega t} \} \; \right\}\;\subset\;{\cal DOEP}^{\omega}_{m}.
\]
Again, we observe that
\[
{\cal DOEP}^{\omega}_{0}\;=\;{\cal DEP}^{\omega}_{0}\;=\;span\left\{1\right\},\qquad{\cal DOEP}^{\omega}_{1}\;=\;{\cal DEP}^{\omega}_{1}\;=\;span\left\{1,\;e^{\omega t},\;e^{-\omega t}\right\},
\]
\[
{\cal D}^2\mathcal{OEP}^{\omega}_{0}\;=\;{\cal D}^2\mathcal{EP}^{\omega}_{0}\;=\;\left\{0\right\},\qquad{\cal D}^2\mathcal{OEP}^{\omega}_{1}\;=\;{\cal D}^2\mathcal{EP}^{\omega}_{1}\;=\;span\left\{e^{\omega t},\;e^{-\omega t}\right\},
\]
but, for $m>1$, ${\cal DOEP}^\omega_m\subsetneq{\cal DEP}^\omega_{2m-1}$ and ${\cal D}^2{\cal OEP}^\omega_m\subsetneq{\cal D}^2{\cal EP}^\omega_{2m-1}$. \\

\begin{definition}[PH curve in ${\cal EP}^{\omega}_{m}$] \label{def:spatial_EPH}
	A parametric curve $\mathbf{r}:[0,1]\rightarrow\RR^d$, $d\in\{2,3\}$, is called a PH curve in ${\cal EP}^{\omega}_{m}$ if and only if one of the following holds
	\begin{itemize}
		\item (planar EPH curve) $d=2$, $\mathbf{r}(t)=(x(t),y(t))$, $x,y\in{\cal EP}^\omega_{m}$, with
		\begin{equation}\label{eq:firstder_2d_Ig}
			\left\{\begin{array}{l}
				x'(t)\;=\;\left(a_1(t)\right)^2\;-\;\left(a_2(t)\right)^2,\\ 
				y'(t)\;=\;2 a_1(t) a_2(t),
			\end{array}\right.
		\end{equation}
		for some exponential polynomials $a_1$, $a_2$.
		\item (spatial EPH curve) $d=3$, $\mathbf{r}(t)=(x(t),y(t),z(t))$, $x,y,z\in{\cal EP}^\omega_{m}$, with
		\begin{equation}\label{eq:firstder_3d_Ig}
			\left\{\begin{array}{l}
				x'(t)\;=\;\left(a_0(t)\right)^2\;+\;\left(a_1(t)\right)^2\;-\;\left(a_2(t)\right)^2\;-\;\left(a_3(t)\right)^2,\\ 
				y'(t)\;=\;2 \left( \;a_1(t) a_2(t) \;+\; a_0(t) a_3(t) \;\right),\\ 
				z'(t)\;=\;2 \left( \;a_1(t) a_3(t)\; -\; a_0(t) a_2(t)\; \right),
			\end{array}\right.
		\end{equation}
		for some exponential polynomials $a_0$, $a_1$, $a_2$, $a_3$.
	\end{itemize}
\end{definition}

\begin{remark} \label{rem:com_root}
	The curves defined via \eqref{eq:firstder_2d_Ig} and \eqref{eq:firstder_3d_Ig} are regular if and only if the exponential polynomials $a_k$ do not have a common root in $[0,1]$ (see, e.g., \cite{Fbook}). 
\end{remark}
In what follows we only consider the case $d=3$, i.e., spatial curves: planar curves ($d=2$) can be easily obtained setting $a_0(t)=a_3(t)=0$ since this condition and \eqref{eq:firstder_3d_Ig} imply \eqref{eq:firstder_2d_Ig} along with $z(t)$ being constant.
From \eqref{eq:firstder_3d_Ig} then, it is easy to get 
$$\left(x'(t)\right)^2\;+\;\left(y'(t)\right)^2\;+\;\left(z'(t)\right)^2\;=\;\left( \;\left(a_0(t)\right)^2\;+\;\left(a_1(t)\right)^2\;+\;\left(a_2(t)\right)^2\;+\;\left(a_3(t)\right)^2\; \right)^2.$$
Thus, the defining characteristic of a PH curve in ${\cal EP}^{\omega}_{m}$ is the fact that the coordinate components of its derivative (or hodograph) comprise a Pythagorean $d$-tuple of functions in ${\cal DEP}^{\omega}_{m}$
- i.e., the sum of their squares coincides with the perfect square of a function in ${\cal DEP}^{\omega}_{m}$.
By virtue of this remarkable property, the parametric speed of the curve $\sigma(t)=|\mathbf{r}'(t)|$ satisfies
$$
\sigma(t)\;=\;\left(a_0(t)\right)^2\;+\;\left(a_1(t)\right)^2\;+\;\left(a_2(t)\right)^2\;+\;\left(a_3(t)\right)^2.
$$
Clearly, a plethora of combinations of exponential polynomials $a_k$ exists so that \eqref{eq:firstder_2d_Ig} and \eqref{eq:firstder_3d_Ig} identify PH curves in ${\cal EP}^{\omega}_m$. In order to simplify both the analysis and the construction, it is easier to identify spaces so that for every choice of the exponential polynomials $a_k$ belonging to such spaces we can guarantee the resulting curves to be PH curves in ${\cal EP}^{\omega}_m$.

\begin{proposition} \label{prop:As}
	Let ${\cal A}^\omega_m$ be either ${\cal DEP}^\omega_{\lfloor m/2\rfloor}$ or ${\cal D}^2{\cal OEP}^{\omega/2}_{\lfloor (m+1)/2\rfloor}$. Then, for every $a_0$, $a_1$, $a_2$, $a_3\in{\cal A}^\omega_m$, equation \eqref{eq:firstder_3d_Ig} defines a PH curve in ${\cal EP}^\omega_m$. 
\end{proposition}
\begin{proof} 
	From \eqref{eq:firstder_3d_Ig}, since ${\cal DEP}^\omega_m$ is a linear space, we only need to prove that
	\[
	\left({\cal DEP}^\omega_{\lfloor m/2\rfloor}\right)^2\;\subseteq\;{\cal DEP}^\omega_{m}\qquad \textrm{ and }\qquad \left({\cal D}^2{\cal OEP}^{\omega/2}_{\lfloor (m+1)/2\rfloor}\right)^2\;\subseteq\;{\cal DEP}^\omega_{m}.
	\]
	Consider then, for $\alpha_k$, $\beta_k\in\RR$, $k=-\lfloor m/2\rfloor,\dots,\lfloor m/2\rfloor$, 
	\[
	\sum_{k=-\lfloor m/2\rfloor}^{\lfloor m/2\rfloor}\alpha_k e^{k\omega t},\;\sum_{k=-\lfloor m/2\rfloor}^{\lfloor m/2\rfloor}\beta_k e^{k\omega t}\;\in\:{\cal DEP}^\omega_{\lfloor m/2\rfloor}.	
	\]
	We have
	\[\begin{array}{c}
		\left(\sum_{k=-\lfloor m/2\rfloor}^{\lfloor m/2\rfloor}\alpha_k e^{k\omega t}\right)\left(\sum_{k=-\lfloor m/2\rfloor}^{\lfloor m/2\rfloor}\beta_k e^{k\omega t}\right)=\sum_{j,k=-\lfloor m/2\rfloor}^{\lfloor m/2\rfloor}\alpha_j\beta_k e^{(j+k)\omega t}\in{\cal DEP}^{\omega}_m,
	\end{array}\]
	since $-m\leq -2\lfloor m/2\rfloor \leq j+k\leq 2\lfloor m/2\rfloor \leq m$.\\ Similarly, for $\alpha_k$, $\beta_k\in\RR$, $k=1-\lfloor (m+1)/2\rfloor,\dots,\lfloor (m+1)/2\rfloor$, and
	\[
	\sum_{k=1-\lfloor (m+1)/2\rfloor}^{\lfloor (m+1)/2\rfloor}\alpha_k e^{\frac{2k-1}{2}\omega t},\;\sum_{k=1-\lfloor (m+1)/2\rfloor}^{\lfloor (m+1)/2\rfloor}\beta_k e^{\frac{2k-1}{2}\omega t}\;\in\:{\cal D}^2{\cal OEP}^{\omega/2}_{\lfloor (m+1)/2\rfloor},
	\]
	we have
	\begin{multline*}
		\left(\sum_{k=1-\lfloor (m+1)/2\rfloor}^{\lfloor (m+1)/2\rfloor}\alpha_k e^{\frac{2k-1}{2}\omega t}\right)\left(\sum_{k=1-\lfloor (m+1)/2\rfloor}^{\lfloor (m+1)/2\rfloor}\beta_k e^{\frac{2k-1}{2}\omega t}\right)\;=\\
		=\;\sum_{j,k=1-\lfloor (m+1)/2\rfloor}^{\lfloor (m+1)/2\rfloor}\alpha_j\beta_k e^{(j+k-1)\omega t}\;\in\;{\cal DEP}^{\omega}_m,
	\end{multline*}
	since $-m\leq 1-2\lfloor (m+1)/2\rfloor \leq j+k-1\leq 2\lfloor (m+1)/2\rfloor-1 \leq m$.
\end{proof}
Examples of major interest that we consider in the following are:\\
\begin{itemize}
	\item $m=1$: \ ${\bf r}(t)$ is a PH curve in ${\cal EP}^{\omega}_{1}=span\{1, t, e^{\omega t}, e^{-\omega t} \}$, its hodograph ${\bf r}'(t)$ is in ${\cal DEP}^{\omega}_{1}=span\{1, e^{\omega t}, e^{- \omega t}\}$ and 
	$\mathcal{A}^\omega_1$ is either ${\cal DEP}^\omega_0=span\{1\}$ or ${\cal D}^2{\cal OEP}^{\omega/2}_{1}=span\{e^{\omega t/2},e^{-\omega t/2}\}$; \\
	
	\item $m=2$: ${\bf r}(t)$ is a PH curve in  ${\cal EP}^{\omega}_{2}=span\{1, t, e^{\omega t}, e^{-\omega t}, e^{2\omega t}, e^{-2\omega t} \}$, its hodograph ${\bf r}'(t)$ is in ${\cal DEP}^{\omega}_{2}=span\{1, e^{\omega t}, e^{-\omega t}, e^{2\omega t}, e^{-2\omega t} \}$ and 
	$\mathcal{A}^\omega_2$ is either ${\cal DEP}^\omega_1=span\{1,e^{\omega t}, e^{-\omega t}\}$ or ${\cal D}^2{\cal OEP}^{\omega/2}_{1}=span\{e^{\omega t/2},e^{-\omega t/2}\}$. \\
\end{itemize}
For $m=1$, $\mathcal{A}^\omega_1={\cal DEP}^\omega_0$ corresponds to trivial line segments. Similarly, for $m=2$, $\mathcal{A}^\omega_2={\cal D}^2{\cal OEP}^{\omega/2}_{1}$ leads to PH curves in ${\cal EP}^\omega_1\subsetneq{\cal EP}^\omega_2$. In general, for $m$ odd, $\mathcal{A}^\omega_m={\cal DEP}^\omega_{\lfloor m/2\rfloor}$ describes curves which actually live in ${\cal EP}^{\omega}_{m-1}$, while, for $m$ even, this is the case for $\mathcal{A}^\omega_m={\cal D}^2{\cal OEP}^{\omega/2}_{\lfloor(m+1)/2\rfloor}$. Therefore, from now on we only consider
\begin{equation} \label{eq:As}
	\mathcal{A}^\omega_m\;:=\;\left\{\begin{array}{cl}
		{\cal DEP}^\omega_{m/2}, & \text{if } m \text{ even}, \\ 
		{\cal D}^2{\cal OEP}^{\omega/2}_{(m+1)/2}, & \text{if } m \text{ odd},
	\end{array}\right.
\end{equation}
for which $\text{dim}(\mathcal{A}^\omega_m)=m+1$. Starting from $\{a_k\}_{k=0}^3$ in $\mathcal{A}^\omega_m$, \eqref{eq:firstder_3d_Ig} defines univocally $x'(t)$, $y'(t)$ and $z'(t)$ that belong to $\mathcal{DEP}^\omega_m$ and finally, by integration, one can obtain the analytic expressions of $x(t)$, $y(t)$ and $z(t)$ in $\mathcal{EP}^\omega_m$. Then, for a fixed $m$, three spaces need to be considered. For each of these spaces a B-basis (see \cite{MP07}) is defined in what follows using the notation summarized in \cref{tab:spaces}.
\begin{table}[h!]
	\caption{Notation used for the exponential-polynomial spaces and their respective B-basis.} \label{tab:spaces}
	\begin{center}\begin{tabular}{|c|c|c|c|}
			\hline
			space	& $\mathcal{A}^\omega_m$ & $\mathcal{DEP}^\omega_m$ & $\mathcal{EP}^\omega_m$\\ 
			\hline
			dimension	& $m+1$ & $2m+1$ & $2(m+1)$\\ 
			\hline
			B-basis	& $\{\psi^\omega_{j,m}\}_{j=0}^m$ & $\{\varphi^\omega_{j,m}\}_{j=0}^{2m}$ & $\{\phi^\omega_{j,m}\}_{j=0}^{2m+1}$\\ 
			\hline
	\end{tabular}\end{center}
\end{table}
According to \cref{def:spatial_EPH}, \cref{rem:com_root}, \cref{prop:As} and \eqref{eq:As}, four functions $a_0, a_1, a_2, a_3$ in $\mathcal{A}^\omega_m$ having no common roots define a PH curve $\mathbf{r}(t)=\left(x(t),y(t),z(t)\right)$ in $\mathcal{EP}^\omega_m$ as in \eqref{eq:firstder_3d_Ig}. Such a curve can thus be associated to a function from the interval $[0,1]$ to the quaternions in a natural way as follows.
\begin{definition} \label{defn:preimage}
	The function
	$$
	\mathbf{A}(t)\;:=\;
	a_0(t) \;+\; a_1(t) \mathbf{i} \;+\; a_2(t) \mathbf{j} \;+\; a_3(t) \mathbf{k},
	$$
	where $\mathbf{i}, \mathbf{j}, \mathbf{k}$ denote the so-called fundamental quaternion units (see, e.g., \cite{Fbook}, Section 5.3), is called the preimage of ${\bf r}(t)$.
\end{definition}
Expanding the coefficients of the preimage with respect to a B-basis of ${\cal A}^{\omega}_{m}$ as
$
a_k(t)=   \sum_{j=0}^{m} a_{k,j} \psi^{\omega}_{j,m}(t)$, $k\in\{0,\dots,3\},
$
for some $a_{k,j}\in\mathbb{R}$, we can rewrite 
\begin{equation}\label{eq:A0A1_spatialg}
	\mathbf{A}(t)\;=\;  \sum_{j=0}^{m} \mathbf{A}_j \psi^{\omega}_{j,m}(t)\qquad\text{where}\qquad\mathbf{A}_j\;=\; a_{0,j} \;+\; a_{1,j} \mathbf{i} \;+\; a_{2,j}  \mathbf{j} \;+\; a_{3,j} \mathbf{k}.
\end{equation}
Moreover, we can compactly write the hodograph $\mathbf{r}'(t)=(x'(t), y'(t), z'(t))$ of $\mathbf{r}(t)$ with the pure vector quaternion
\begin{equation}\label{eq:hodograph_quaternionic_formg}
	\mathbf{r}'(t)\;=\;
	\mathbf{A}(t) \mathbf{i} \mathbf{A}^*(t).
\end{equation}
Here and in the following, with an abuse of notation, we identify vectors in $\mathbb{R}^3$ and pure vector quaternions via the natural bijection
$
(x,y,z)\;\longleftrightarrow\; x\mathbf{i}+y\mathbf{j}+z\mathbf{k}.
$
Accordingly, in view of \eqref{eq:hodograph_quaternionic_formg}, the parametric speed of $\mathbf{r}(t)$ has the quaternionic representation
\begin{equation}\label{eq:speed_quaternionic_formg}
	\sigma(t)\;=\;|\mathbf{r'}(t)|\;=\;|\mathbf{A}(t) \mathbf{i} \mathbf{A}^*(t)|\;=\;|\mathbf{A}(t)|^2\;=\;\mathbf{A}(t) \mathbf{A}^*(t).
\end{equation}

\section{PH curves in ${\cal EP}^{\omega}_1$}\label{sec3}

\subsection{The normalized B-basis of the space ${\cal EP}^{\omega}_1$}
On the interval $[0,1]$ the non-negative exponential functions
\begin{equation}
	\begin{array}{l}
		\psi_{0,1}^{\omega}(t)\;=\;  \frac{e^{\frac{\omega}{2}(1-t)}-e^{-\frac{\omega}{2}(1-t)}}{e^\frac{\omega}{2}-e^{-\frac{\omega}{2}}}
		\;=\;\frac{ \sinh\left( \frac{\omega}{2}-\frac{\omega}{2}t \right)}{\sinh\left(\frac{\omega}{2}\right)}, \\ 
		\psi_{1,1}^{\omega}(t)\;=\;  \frac{e^{\frac{\omega}{2}t}-e^{-\frac{\omega}{2} t}}{e^\frac{\omega}{2}-e^{-\frac{\omega}{2}}}
		\;=\;\frac{\sinh\left(\frac{\omega}{2}t\right)}{\sinh\left(\frac{\omega}{2}\right)}
	\end{array}
\end{equation}
define a B-basis of the extended Chebyshev space \scalebox{0.97}{${\cal A}^{\omega}_1=\mathcal{D}^2\mathcal{OEP}^{\omega/2}_1=span\{ e^{\frac{\omega}{2}t}, e^{-\frac{\omega}{2}t} \}$}.
However, note that $\{\psi_{0,1}^{\omega}(t), \psi_{1,1}^{\omega}(t)\}$ is not normalized since $\psi_{0,1}^{\omega}(t)+\psi_{1,1}^{\omega}(t) \neq 1$ for $t \in (0,1)$. By squaring an arbitrary function $f \in {\cal A}^{\omega}_1$, we obtain a function $f^2$ that belongs to the exponential space
${\cal DEP}^{\omega}_1=span\{ 1, e^{\omega t}, e^{-\omega t} \}$.
Since we assume $\omega \in \RR^+$, ${\cal DEP}^{\omega}_1$ is an extended Chebyshev space that, on the interval $[0,1]$, admits a normalized B-basis of the form
\begin{equation}\label{eq:U2basis}
	\resizebox{0.915\textwidth}{!}{$\begin{array}{l}
			\varphi_{0,1}^{\omega}(t)\;=\;  \frac{(e^{\omega(1-t)} - 1)^2 \, e^{\omega t}}{(e^{\omega} - 1)^2}\;=\;
			\frac{\cosh(\omega-\omega t)-1}{\cosh(\omega)-1}, \\ 
			\varphi_{1,1}^{\omega}(t)\;=\;  \frac{(e^{-\omega t} - 1) \, (e^{\omega t} - e^{\omega}) \, (e^{\omega} + 1)}{(e^{\omega} - 1)^2}\;=\;
			\frac{\cosh(\omega) -\cosh(\omega t) -\cosh(\omega-\omega t) + 1}{\cosh(\omega)-1}, \\ 
			\varphi_{2,1}^{\omega}(t)\;=\;  \frac{(e^{\omega t}-1)^2 \, e^{\omega(1-t)}}{(e^{\omega} - 1)^2}\;=\;
			\frac{\cosh(\omega t)-1}{\cosh(\omega)-1}.
		\end{array}$}
\end{equation}
The exponential functions $\{\varphi_{0,1}^{\omega}(t),\varphi_{1,1}^{\omega}(t),\varphi_{2,1}^{\omega}(t)\}$ satisfy the following relationships with the exponential functions $\{\psi_{0,1}^{\omega}(t),\psi_{1,1}^{\omega}(t)\}$:
\begin{equation}\label{equivalenze1}
	\begin{array}{c}
		\resizebox{0.88\textwidth}{!}{$\left( \psi_{0,1}^{\omega}(t) \right)^2 = \varphi_{0,1}^{\omega}(t), \quad
			\psi_{0,1}^{\omega}(t)  \psi_{1,1}^{\omega}(t) = \frac{1}{2} { c_1(\omega)}  \varphi_{1,1}^{\omega}(t),  \quad
			\left( \psi_{1,1}^{\omega}(t) \right)^2 = \varphi_{2,1}^{\omega}(t),$}
		\\  \\
		\quad \hbox{with}
		\quad
		{c_1(\omega)}\;=\;   \frac{2}{e^{\frac{\omega}{2}}+e^{-\frac{\omega}{2}}} \;=\;\frac{1}{\cosh\left(\frac{\omega}{2}\right)}.
	\end{array}
\end{equation}
The antiderivative of $f^2 \in {\cal DEP}^{\omega}_1$ is an exponential-polynomial function that belongs to the order-4 exponential-polynomial space ${\cal EP}^{\omega}_1=span \{ 1, t, e^{\omega t}, e^{-\omega t} \}$.
The exponential-polynomial functions
\begin{equation}\label{eq:U3basis}
	\resizebox{0.915\textwidth}{!}{$\begin{array}{lll}
			\phi_{0,1}^{\omega}(t)&=&  \frac{e^{\omega(1-t)}-e^{-\omega(1-t)}-2\omega(1-t)}{e^{\omega}-e^{-\omega}-2\omega} \;=\; \frac{\sinh(\omega-\omega t)-\omega(1-t)}{\sinh(\omega)-\omega}, \\ 
			\phi_{1,1}^{\omega}(t)&=& \frac{(e^{-\omega}-1) \, \Big( (\omega+1-e^{\omega}) e^{\omega t}+((\omega-1)e^{\omega}+1) e^{\omega(1-t)} \Big)+
				(e^{\omega}-1) \, \Big( \omega \left(2t+(e^{-\omega}+e^{\omega})(1-t) \right)+e^{-\omega}-e^{\omega} \Big)}{(\omega+2+e^{\omega}(\omega-2)) \, (e^{\omega}-e^{-\omega}-2\omega)} \\
			&=&  \frac{-\omega t -\omega(1-t) \cosh(\omega) + \omega \cosh(\omega-\omega t) + \sinh(\omega) - \sinh(\omega t) -\sinh(\omega-\omega t)}{(\omega \coth\left(\frac{\omega}{2} \right)-2) (\omega-\sinh(\omega))}, \\
			\phi_{2,1}^{\omega}(t)&=& \frac{(e^{-\omega}-1) \, \Big( (\omega+1-e^{\omega}) e^{\omega(1-t)}+((\omega-1)e^{\omega}+ 1) e^{\omega t} \Big)+
				(e^{\omega}-1) \, \Big( \omega \left(2(1-t)+(e^{-\omega}+e^{\omega})t\right)+e^{-\omega}-e^{\omega} \Big)}{(\omega+2+e^{\omega}(\omega-2)) \, (e^{\omega}-e^{-\omega}-2\omega)}  \\
			&=& 
			\frac{-\omega(1-t) -\omega t \cosh(\omega) + \omega \cosh(\omega t) + \sinh(\omega) -\sinh(\omega - \omega t) -\sinh(\omega t)}{(\omega  \coth\left(\frac{\omega}{2} \right)-2) (\omega - \sinh(\omega))}, \\
			\phi_{3,1}^{\omega}(t)&=&  \frac{e^{\omega t}-e^{-\omega t}-2\omega t}{e^{\omega}-e^{-\omega}-2\omega}\; =\;
			\frac{\sinh(\omega t)-\omega t}{\sinh(\omega)-\omega},
		\end{array}$}
\end{equation}
define a normalized B-basis of the extended Chebyshev space ${\cal EP}^{\omega}_1$ on $[0,1]$. For later use we observe that for the antiderivatives of the basis functions of ${\cal DEP}^{\omega}_1$ we can write
\begin{equation}\label{integrali1}
	\begin{array}{c}
		\int_0^t{\varphi_{0,1}^{\omega}(x)  dx} \;=\;
		{		c_2(\omega) \;\sum_{i=1}^3\;\phi_{i,1}^{\omega}(t)},\quad 
		\int_0^t{\varphi_{1,1}^{\omega}(x)  dx} \;=\; 
		{	\frac{c_3(\omega)}{c_1(\omega)} \; \sum_{i=2}^3\;\phi_{i,1}^{\omega}(t)},\\  \\
		\int_0^t{\varphi_{2,1}^{\omega}(x)  dx} \;=\; c_2(\omega) \; \phi_{3,1}^{\omega}(t),
	\end{array}
\end{equation}
with
\begin{equation}\label{eq:cs}
	\begin{array}{c}
		c_2(\omega) {\;=\; \int_0^1\;\varphi^\omega_{0,1}(t)\;dt} \;=\;   \frac{\sinh(\omega)-\omega}{\omega  (\cosh(\omega)-1)},\\ \\
		c_3(\omega) {\;=\; c_1(\omega)\;\int_0^1\;\varphi^\omega_{1,1}(t)\;dt} \;=\; \frac{\frac{\omega}{2}  \coth\left( \frac{\omega}{2} \right)-1}{\frac{\omega}{2}  \sinh\left( \frac{\omega}{2} \right)}.
	\end{array}
\end{equation}

{
	\begin{remark}
		For all $\omega \in \mathbb{R}^+$, we always have $c_1(\omega), c_2(\omega), c_3(\omega) \neq 0$.
	\end{remark}
}

\subsection{Geometric properties of B\'ezier-like curves in ${\cal EP}^{\omega}_1$}

\begin{definition}[B\'ezier-like curves in ${\cal EP}^{\omega}_1$]\label{def:UEcurve}
	Given a control polygon with vertices ${\bf r}_i \in \RR^d$, $i=0,\dots,3$, the associated B\'ezier-like curve in ${\cal EP}^{\omega}_1$ is defined as
	\begin{equation} \label{eq:cub_Bez}
		{\bf r}(t)=\sum_{i=0}^3 \, {\bf r}_i \, \phi_{i,1}^{\omega}(t), \qquad t\in[0,1].
	\end{equation}
\end{definition}

\begin{proposition}[Properties of B\'ezier-like curves in ${\cal EP}^{\omega}_1$]\label{prop:EP4} \newline
	The B\'ezier-like curve in \eqref{eq:cub_Bez} has the following properties:
	\begin{itemize}
		\item[(a)] \emph{Convex hull property and geometric invariance property}. The entire curve lies inside the convex hull of its control points and its shape is independent of the coordinate system, i.e., it is scale and translation invariant.
		\item[(b)] \emph{Symmetry}. The control points ${\bf r}_0, \, {\bf r}_1, \, {\bf r}_2, \, {\bf r}_3$ and ${\bf r}_3, \, {\bf r}_{2}, \, {\bf r}_1, \, {\bf r}_0$ define the same curve with respect to different parameterizations, i.e., \[\sum_{i=0}^3 \, {\bf r}_i \phi_{i,1}^{\omega}(t)\;=\;\sum_{i=0}^3 \, {\bf r}_{3-i} \phi_{i,1}^{\omega}(1-t), \qquad t \in [0,1].\]
		\item[(c)]  \emph{Derivative formula}. 
		$$
		\frac{d}{dt}{\bf r}(t)\;=\;\sum_{i=0}^{2} \, 
		{\frac{\Delta{\bf r}_i}{\int_{0}^1\;\varphi_{i,1}^\omega(x)\;dx} } \, \varphi_{i,1}^{\omega}(t), \qquad t \in[0,1],
		$$
		where, for all $i=0,1,2$, $\Delta{\bf r}_i={\bf r}_{i+1}-{\bf r}_i$.
		\item[(d)] \emph{Endpoint conditions.}
		$$
		\begin{array}{cc}
			{\bf r}(0)\;=\;{\bf r}_0,\qquad  & {\bf r}'(0)\;=\;
			\frac{1}{c_2(\omega)} \, ({\bf r}_1-{\bf r}_0),\\ \\
			{\bf r}(1)\;=\;{\bf r}_3, \qquad & {\bf r}'(1)\;=\;
			\frac{1}{c_2(\omega)} \, ({\bf r}_{3}-{\bf r}_{2}).
		\end{array}
		$$
	\end{itemize}
\end{proposition}
\begin{proof}
	\Cref{SM:prop:EP4}.
\end{proof}

\subsection{Control polygons of PH curves in ${\cal EP}^{\omega}_1$}

To construct a PH curve in ${\cal EP}^{\omega}_1$, the functions  $a_0,a_1,a_2,a_3$ are chosen in ${\cal A}^{\omega}_1$ and thus
\[
a_k(t)\;=\;a_{k,0}\psi_{0,1}^\omega(t)\;+\; a_{k,1} \psi_{1,1}^{\omega}(t),\qquad k\in\{0,\dots,3\},
\]
for some $a_{k,0}$, $a_{k,1}\in\mathbb{R}$.
Consequently, the associated preimage is
\begin{equation} \label{eq:preimage_ep4}
	\begin{array}{l}
		\mathbf{A}(t)\;=\;\mathbf{A}_0 \psi_{0,1}^{\omega}(t) \;+\; \mathbf{A}_1 \psi_{1,1}^{\omega}(t)
	\end{array}
\end{equation}
where
\begin{equation}\label{eq:Aj_spatial_cubic}
	\mathbf{A}_j\;=\; a_{0,j}\; +\; a_{1,j} \mathbf{i} \;+\; a_{2,j} \mathbf{j} \;+\; a_{3,j} \mathbf{k}, 
	\qquad j\;=\;0,\;1.
\end{equation}

\begin{proposition} \label{prop:Bez_EP1}
	A PH curve $\mathbf{r}(t)$ in ${\cal EP}^{\omega}_1$ can be expressed in the B\'ezier-like form
	$
	\mathbf{r}(t)=\sum_{i=0}^3 \mathbf{r}_i \phi_{i,1}^{\omega}(t),
	$
	with B\'ezier-like control points $\mathbf{r}_i$, $i=1,\ldots,3$ given in terms of the freely chosen integration constant $\mathbf{r}_0$, of the numbers in \eqref{eq:cs} and of the coefficients of the preimage $\mathbf{A}(t)$ in \eqref{eq:preimage_ep4} and \eqref{eq:Aj_spatial_cubic} by
	\begin{equation}\label{eq:cps3d}
		\begin{array}{c}
			\mathbf{r}_1 \;=\; \mathbf{r}_0 \;+\; c_2(\omega) \, \mathbf{A}_0\mathbf{i}\mathbf{A}_0^*,  \qquad
			\mathbf{r}_2 \;=\; \mathbf{r}_1 \;+\; c_3(\omega) \, \frac{1}{2} \, (\mathbf{A}_0\mathbf{i}\mathbf{A}_1^*+\mathbf{A}_1\mathbf{i}\mathbf{A}_0^*),  \\ \\
			\mathbf{r}_3 \;=\; \mathbf{r}_2 \;+\; c_2(\omega) \,  \mathbf{A}_1\mathbf{i}\mathbf{A}_1^*.
		\end{array}
	\end{equation}
\end{proposition}
\begin{proof}
	\Cref{SM:prop:Bez_EP1}.
\end{proof}

\begin{remark}
	Since, from 
	\eqref{eq:cs}, $\lim_{\omega \rightarrow 0}c_2(\omega)=\lim_{\omega \rightarrow 0}c_3(\omega)=1/3$, 
	\eqref{eq:cps3d} 
	recover the well-known results of the cubic polynomial case when $\omega \rightarrow 0$, see \cite{Fbook}.
\end{remark}

\subsection{Parametric speed and arc length in ${\cal EP}^{\omega}_1$}

\begin{proposition} \label{prop:par_speed_EP1}
	The parametric speed of $\mathbf{r}(t)$ is a function in ${\cal DEP}^{\omega}_1$ with the explicit expression
	$
	\sigma(t)= \sum_{i=0}^2 \sigma_i  \varphi_{i,1}^{\omega}(t),
	$
	where
	\begin{equation}\label{sigmacoeff_2}
		\begin{array}{c}
			\sigma_0\;=\;|\mathbf{A}_0|^2,\qquad
			\sigma_1\;=\; { c_1(\omega)} \, \frac{1}{2}  \, (\mathbf{A}_1 \mathbf{A}_0^* + \mathbf{A}_0 \mathbf{A}_1^*),\qquad
			\sigma_2\;=\;|\mathbf{A}_1|^2.
		\end{array}
	\end{equation}
\end{proposition}
\begin{proof}
	\Cref{SM:prop:par_speed_EP1}.
\end{proof}

\begin{proposition} \label{prop:arc_EP1}
	The arc length function of $\mathbf{r}(t)$ is a function in ${\cal EP}^{\omega}_1$ having the expression
	$
	s(t)= \sum_{i=0}^3 s_i  \phi_{i,1}^{\omega}(t),
	$
	where
	$$
	{
		s_0\;=\; 0},
	\qquad
	s_1\;=\; s_0 + \sigma_0 c_2(\omega), \qquad
	s_2\;=\; s_1 +  \sigma_1 \frac{c_3(\omega)}{c_1(\omega)}, \qquad
	s_3\;=\; s_2 + \sigma_2 c_2(\omega).
	$$
\end{proposition}
\begin{proof}
	\Cref{SM:prop:arc_EP1}.
\end{proof}

\begin{corollary}
	The total arc length  of $\mathbf{r}(t)$ is
	\begin{equation}\label{totarclength}
		\begin{array}{rcl}
			{ s(1)}&=&{s_3}\;=\;(\sigma_0+\sigma_2) \, c_2(\omega) \;+\; \sigma_1 \, \frac{c_3(\omega)}{c_1(\omega)} \\ \\
			&=& c_2(\omega) \, \left( |\mathbf{A}_0|^2 + |\mathbf{A}_1|^2 \right)
			\;+\; c_3(\omega) \, \frac{1}{2} \, (\mathbf{A}_1 \mathbf{A}_0^* + \mathbf{A}_0 \mathbf{A}_1^*).
		\end{array}
	\end{equation}
\end{corollary}

\begin{remark}
	When $\omega \rightarrow 0$, the total arc length formula in \eqref{totarclength} yields $s(1)= \left( |\mathbf{A}_0|^2 + |\mathbf{A}_1|^2 \right)/3 +\, (\mathbf{A}_1 \mathbf{A}_0^* + \mathbf{A}_0 \mathbf{A}_1^*)/6$, thus recovering the well-known result of the cubic polynomial case \cite{Fbook}.
\end{remark}


\section{PH curves in ${\cal EP}^{\omega}_2$}\label{sec4}

\subsection{The normalized B-basis of the space ${\cal EP}^{\omega}_2$}

By squaring an arbitrary function $f \in {\cal A}^{\omega}_2=span \{ 1, e^{\omega t}, e^{-\omega t} \}$, we obtain a function $f^2$ that belongs to the exponential space ${\cal DEP}^{\omega}_2=span \{ 1, e^{\omega t}, e^{-\omega t}, e^{2\omega t}, e^{-2\omega t} \}$. Since ${\cal A}^{\omega}_2=\mathcal{DEP}^\omega_1$ we can choose $\psi^\omega_{j,2}(t)=\varphi^\omega_{j,1}(t)$, $j\in\{0,1,2\}$ as in \eqref{eq:U2basis}.
Then, ${\cal DEP}^{\omega}_2$ is an extended Chebyshev space that, on the interval $[0,1]$, admits a normalized B-basis of the form
$$
\begin{array}{c}
	\varphi_{0,2}^{\omega}(t)\;=\;\left( \psi_{0,2}^{\omega}(t) \right)^2, \qquad
	\varphi_{1,2}^{\omega}(t)\;=\;2 \psi_{0,2}^{\omega}(t) \psi_{1,2}^{\omega}(t), \\ \\
	\varphi_{2,2}^{\omega}(t)\;=\; \left( \psi_{1,2}^{\omega}(t) \right)^2 \;+\; 2 \psi_{0,2}^{\omega}(t) \psi_{2,2}^{\omega}(t), \\ \\
	\varphi_{3,2}^{\omega}(t)\;=\;2 \psi_{1,2}^{\omega}(t) \psi_{2,2}^{\omega}(t), \qquad
	\varphi_{4,2}^{\omega}(t)\;=\; \left( \psi_{2,2}^{\omega}(t) \right)^2.
\end{array}
$$
The inverse relationship between
the exponential functions $\{\varphi^\omega_{j,2}(t)\}_{j=0}^{4}$ 
and the exponential functions $\{\psi^\omega_{j,2}(t)\}_{j=0}^{2}$  
is instead given by
\begin{equation}\label{equivalenze2}
	\begin{array}{c}
		\left( \psi_{0,2}^{\omega}(t) \right)^2\;=\; \varphi_{0,2}^{\omega}(t), \qquad
		\psi_{0,2}^{\omega}(t) \psi_{1,2}^{\omega}(t) \;=\; \frac{1}{2} \varphi_{1,2}^{\omega}(t), \\ \\
		\left( \psi_{1,2}^{\omega}(t) \right)^2\;=\; q_0(\omega)  \varphi_{2,2}^{\omega}(t), \qquad
		\psi_{0,2}^{\omega}(t) \psi_{2,2}^{\omega}(t) \;=\; \frac{1}{2} q_1(\omega)  \varphi_{2,2}^{\omega}(t), \\ \\
		\psi_{1,2}^{\omega}(t) \psi_{2,2}^{\omega}(t) \;=\; \frac{1}{2} \varphi_{3,2}^{\omega}(t), \qquad
		\left( \psi_{2,2}^{\omega}(t) \right)^2\;=\; \varphi_{4,2}^{\omega}(t), \\ \\
		\qquad \text{with} \qquad
		q_0(\omega)\;=\;\frac{\cosh(\omega) + 1}{\cosh(\omega) + 2}, \qquad q_1(\omega)\;=\;\frac{1}{\cosh(\omega) + 2}.
	\end{array}
\end{equation}
The antiderivative of a function in ${\cal DEP}^{\omega}_2$ is an exponential-polynomial function that belongs to the order-6 exponential-polynomial space ${\cal EP}^{\omega}_2$. 
The exponential-polynomial functions
\begin{equation}\label{eq:U5basis}
	\resizebox{0.915\textwidth}{!}{$\begin{array}{c}
			\phi_{0,2}^{\omega}(t)\;=\;\frac{g_0(\omega-\omega t)}{g_0(\omega)} , \qquad
			
			\phi_{1,2}^{\omega}(t)\;=\;g_1(\omega) \sinh\left(\frac{\omega}{2} \right) \, \left( \sinh^4 \left(\frac{\omega-\omega t}{2} \right) - \sinh^4 \left(\frac{\omega}{2} \right) \, \frac{g_0(\omega-\omega t)}{g_0(\omega)} \right), \\ \\
			
			\phi_{2,2}^{\omega}(t)\;=\;g_2(\omega)  \, \left( -16 \sinh^3 \left(\frac{\omega-\omega t}{2} \right) \, \sinh \left(\frac{\omega t}{2} \right) +
			g_1(\omega) g_0(\omega) \sinh^4 \left(\frac{\omega-\omega t}{2} \right) - g_1(\omega) \sinh^4 \left(\frac{\omega}{2} \right) \, g_0(\omega-\omega t) \right), \\ \\
			
			\phi_{3,2}^{\omega}(t)\;=\;g_2(\omega)  \, \left( -16 \sinh^3 \left(\frac{\omega t}{2} \right) \, \sinh \left(\frac{\omega-\omega t}{2} \right) +
			g_1(\omega) g_0(\omega) \sinh^4 \left(\frac{\omega t}{2} \right) - g_1(\omega) \sinh^4 \left(\frac{\omega}{2} \right) \, g_0(\omega t) \right), \\ \\
			
			\phi_{4,2}^{\omega}(t)\;=\;g_1(\omega) \sinh \left(\frac{\omega}{2} \right) \, \left( \sinh^4 \left(\frac{\omega t}{2} \right) -\sinh^4 \left(\frac{\omega}{2} \right) \, \frac{g_0(\omega t)}{g_0(\omega)}  \right), \qquad
			
			\phi_{5,2}^{\omega}(t)\;=\;\frac{g_0(\omega t)}{g_0(\omega)},
		\end{array}$}
\end{equation}
with
$$
\begin{array}{c}
	g_0(\omega)=3 \omega + \sinh(\omega)  (\cosh(\omega)-4), \qquad
	
	g_1(\omega)=  \frac{4}{ \sinh(\frac{\omega}{2}) \left( \cosh(\omega)- 3 \omega \coth(\frac{\omega}{2})+5 \right)  }, \\ \\
	
	g_2(\omega)\;=\; \frac{\sinh(\frac{\omega}{2})}{3 \big(3 \sinh(\omega) - \omega(\cosh(\omega)+2) \big)},
\end{array}
$$
define a normalized B-basis of the extended Chebyshev space ${\cal EP}^{\omega}_2$ on $[0,1]$. For later use we observe that for the antiderivatives of the basis functions of ${\cal DEP}^{\omega}_2$ we can write
\begin{equation}\label{integrali2}
	\begin{array}{c}
		\int_0^t \varphi_{0,2}^{\omega}(x) dx\;=\; {q_2(\omega) \; \sum_{i=1}^{5} \phi_{i,2}^{\omega}(t)},  \quad
		
		\int_0^t \varphi_{1,2}^{\omega}(x) dx\;=\; {q_3(\omega) \; \sum_{i=2}^{5} \phi_{i,2}^{\omega}(t)},  \\ \\
		
		\int_0^t \varphi_{2,2}^{\omega}(x) dx\;=\;  {\frac{q_4(\omega)}{q_1(\omega)} \; \sum_{i=3}^{5} \phi_{i,2}^{\omega}(t)},  \\  \\
		
		\int_0^t \varphi_{3,2}^{\omega}(x) dx\;=\; {q_3(\omega) \; \sum_{i=4}^{5} \phi_{i,2}^{\omega}(t)},\quad 
		
		\int_0^t \varphi_{4,2}^{\omega}(x) dx\;=\; q_2(\omega)\; \phi_{5,2}^{\omega}(t),
	\end{array}
\end{equation}
with
\begin{equation}\label{ghj}
	\begin{array}{c}
		q_2(\omega) {\;=\; \int_{0}^1\;\varphi^\omega_{0,2}(t)\;dt } \;=\; \frac{g_0(\omega)}{2 \omega (\cosh(\omega) - 1)^2}, \\ \\
		
		q_3(\omega) {\;=\; \int_{0}^1\;\varphi^\omega_{1,2}(t)\;dt }\;=\;   \frac{5 \sinh(\omega)-3\omega +(\sinh(\omega)-3\omega) \cosh(\omega)}{ \omega (\cosh(\omega) - 1)^2}, \\ \\
		
		q_4(\omega) {\;=\; q_1(\omega)\;\int_{0}^1\;\varphi^\omega_{2,2}(t)\;dt } \;=\; \frac{\omega (2+\cosh(\omega)) -3 \sinh(\omega)}{\omega (\cosh(\omega) - 1)^2}.
	\end{array}
\end{equation}

{
	\begin{remark}
		For all $\omega \in \mathbb{R}^+$, we always have $g_0(\omega), g_1(\omega), g_2(\omega) \neq 0$ as well as $q_0(\omega), q_1(\omega), q_2(\omega), q_3(\omega), q_4(\omega) \neq 0$.
	\end{remark}
}


\subsection{Geometric properties of B\'ezier-like curves in ${\cal EP}^{\omega}_2$}

\begin{definition}[B\'ezier-like curves in ${\cal EP}^{\omega}_2$]\label{def:UEcurveII}
	Given a control polygon with vertices ${\bf r}_i \in \RR^d$, $i=0,\dots,5$, the associated B\'ezier-like curve in ${\cal EP}^{\omega}_2$ is defined as
	\begin{equation} \label{eq:quint_Bez}
		{\bf r}(t)\;=\;\sum_{i=0}^5 \, {\bf r}_i \, \phi_{i,2}^{\omega}(t), \qquad t\in[0,1].
	\end{equation}
\end{definition}

\begin{proposition}[Properties of B\'ezier-like curves in ${\cal EP}^{\omega}_2$]\label{prop:EP6} \newline The B\'ezier-like curve in \eqref{eq:quint_Bez} has the following properties:
	\begin{itemize}
		\item[(a)] \emph{Convex hull property and geometric invariance property}. The entire curve lies inside the convex hull of its control points and its shape is independent of the coordinate system, i.e., it is scale and translation invariant.
		\item[(b)] \emph{Symmetry}. The control points ${\bf r}_0, \, {\bf r}_1, \, ..., \, {\bf r}_5$ and ${\bf r}_5, \, ..., \, {\bf r}_1, \, {\bf r}_0$ define the same curve with respect to different parameterizations, i.e., 
		\[
		\sum_{i=0}^5 \, {\bf r}_i \phi_{i,2}^{\omega}(t)\;=\;\sum_{i=0}^5 \, {\bf r}_{5-i} \phi_{i,2}^{\omega}(1-t), \qquad  t \in [0,1].
		\]
		\item[(c)] \emph{Derivative formula}.
		$$
		\frac{d}{dt}{\bf r}(t)\;=\;\sum_{i=0}^{4} \, 
		{\frac{\Delta{\bf r}_i}{\int_{0}^1\;\varphi^\omega_{i,2}(x)\;dx} } \, \varphi_{i,2}^{\omega}(t), \qquad  t \in[0,1],
		$$
		where, for all $i=0,\ldots,4$, $\Delta{\bf r}_i={\bf r}_{i+1}-{\bf r}_i$. 
		\item[(d)] \emph{Endpoint conditions.}
		$$
		\begin{array}{cc}
			{\bf r}(0)\;=\;{\bf r}_0, \qquad & {\bf r}'(0)\;=\; 
			\frac{1}{q_2(\omega)} ({\bf r}_1-{\bf r}_0),\\ \\
			{\bf r}(1)\;=\;{\bf r}_5, \qquad & {\bf r}'(1)\;=\;
			\frac{1}{q_2(\omega)} ({\bf r}_{5}-{\bf r}_{4}).
		\end{array}
		$$
	\end{itemize}
\end{proposition}
\begin{proof}
	\Cref{SM:prop:EP6}.
\end{proof}


\subsection{Control polygons of PH curves in ${\cal EP}^{\omega}_2$}

To construct a spatial PH curve in ${\cal EP}^{\omega}_2$, the functions $a_0,a_1,a_2,a_3$ are chosen in ${\cal A}^{\omega}_2$ and thus
$
a_k(t)=\sum_{j=0}^2a_{k,j}\psi^\omega_{j,2}(t)$, $k\in\{0,\dots,3\},
$
for some $a_{k,j}\in\mathbb{R}$. Consequently, the associated preimage is
\begin{equation}\label{eq:At_spatial_quintic}
	\begin{array}{l}
		\mathbf{A}(t)\;=\;\sum_{j=0}^2\;\mathbf{A}_j\psi^\omega_{j,2}(t)
	\end{array}
\end{equation}
where
\begin{equation}\label{eq:Aj_spatial_quintic}
	\mathbf{A}_j\;=\; a_{0,j} \;+\; a_{1,j} \mathbf{i} \;+\; a_{2,j} \mathbf{j} \;+\; a_{3,j} \mathbf{k}, \qquad  j=0,1,2.
\end{equation}
\begin{proposition} \label{prop:Bez_EP2}
	A PH curve $\mathbf{r}(t)$ in ${\cal EP}^{\omega}_2$ can be expressed in the B\'ezier-like form
	$
	\mathbf{r}(t)=\sum_{i=0}^5 \mathbf{r}_i \phi_{i,2}^{\omega}(t),
	$
	with B\'ezier-like control points $\mathbf{r}_i$, $i=1,\dots,5$ given in terms of the freely chosen integration constant $\mathbf{r}_0$, of the numbers in \eqref{equivalenze2} and \eqref{ghj} and  of the coefficients of the preimage $\mathbf{A}(t)$ in \eqref{eq:At_spatial_quintic} and \eqref{eq:Aj_spatial_quintic} by
	\begin{equation}\label{eq:cps3dII}
		\begin{array}{c}
			\mathbf{r}_1 \;=\; \mathbf{r}_0 \;+\; q_2(\omega) \, \mathbf{A}_0\mathbf{i}\mathbf{A}_0^*, \qquad
			
			\mathbf{r}_2 \;=\; \mathbf{r}_1 \;+\; q_3(\omega) \, \frac{1}{2} \, (\mathbf{A}_0\mathbf{i}\mathbf{A}_1^*+\mathbf{A}_1\mathbf{i}\mathbf{A}_0^*), \\ \\
			
			\mathbf{r}_3 \;=\; \mathbf{r}_2 \;+\; q_4(\omega) \, \left( \frac{1}{2}(\mathbf{A}_0\mathbf{i}\mathbf{A}_2^*+\mathbf{A}_2\mathbf{i}\mathbf{A}_0^*)+
			\frac{q_0(\omega)}{q_1(\omega)} \,   \mathbf{A}_1\mathbf{i}\mathbf{A}_1^* \right),  \\ \\
			
			\mathbf{r}_4 \;=\; \mathbf{r}_3 \;+\; q_3(\omega) \, \frac{1}{2} \, (\mathbf{A}_1\mathbf{i}\mathbf{A}_2^*+\mathbf{A}_2\mathbf{i}\mathbf{A}_1^*),  \qquad
			
			\mathbf{r}_5 \;=\; \mathbf{r}_4 \;+\; q_2(\omega) \, \mathbf{A}_2\mathbf{i}\mathbf{A}_2^*.
		\end{array}
	\end{equation}
\end{proposition}
\begin{proof}
	\Cref{SM:prop:Bez_EP2}.
\end{proof}

\begin{remark}
	Since, from \eqref{equivalenze2} and \eqref{ghj}, 
	\[
	\lim_{\omega \rightarrow 0} \frac{q_0(\omega)}{q_1(\omega)}\;=\;2,\qquad\lim_{\omega \rightarrow 0} q_2(\omega)\;=\; \lim_{\omega \rightarrow 0} q_3(\omega)  \;=\frac{1}{5}\qquad \text{and}\qquad\lim_{\omega \rightarrow 0} q_4(\omega)\;=\;\frac{1}{15},
	\] 
	\eqref{eq:cps3dII} recover the well-known results of the quintic polynomial case when $\omega \rightarrow 0$, see \cite{Fbook}.
\end{remark}


\subsection{Parametric speed and arc length in ${\cal EP}^{\omega}_2$}

\begin{proposition} \label{prop:par_speed_EP2}
	The parametric speed of $\mathbf{r}(t)$ is a function in ${\cal DEP}^{\omega}_2$ having the expression
	$
	\sigma(t)= \sum_{i=0}^4 \sigma_i  \varphi_{i,2}^{\omega}(t), 
	$
	where
	\begin{equation}\label{sigmacoeff_2II}
		\begin{array}{c}
			\sigma_0\;=\; |\mathbf{A}_0|^2,\qquad\sigma_1\;=\; \frac{1}{2}(\mathbf{A}_1 \mathbf{A}_0^* + \mathbf{A}_0 \mathbf{A}_1^*),\\ \\
			\sigma_2\;=\; q_0(\omega) \, |\mathbf{A}_1|^2 \;+\; q_1(\omega) \, \frac{1}{2} \, (\mathbf{A}_2 \mathbf{A}_0^* + \mathbf{A}_0 \mathbf{A}_2^*),\\ \\
			\sigma_3\;=\; \frac{1}{2}(\mathbf{A}_1 \mathbf{A}_2^* + \mathbf{A}_2 \mathbf{A}_1^*),\qquad\sigma_4\;=\; |\mathbf{A}_2|^2.
		\end{array}
	\end{equation}
\end{proposition}
\begin{proof}
	\Cref{SM:prop:par_speed_EP2}.
\end{proof}

\begin{proposition} \label{prop:arc_EP2}
	The arc length function of $\mathbf{r}(t)$ is a function in ${\cal EP}^{\omega}_2$ having the expression
	$
	s(t)= \sum_{i=0}^5 s_i \phi_{i,2}^{\omega}(t),
	$
	where
	$$
	{
		s_0\;=\; 0,}\qquad
	s_1\;=\; s_0 \;+\;  \sigma_0 q_2(\omega), \qquad
	s_2\;=\; s_1 \;+\; \sigma_1 q_3(\omega), 
	$$
	$$
	s_3\;=\; s_2 \;+\; \sigma_2 \frac{q_4(\omega)}{q_1(\omega)}, \qquad
	s_4\;=\; s_3 \;+\; \sigma_3 q_3(\omega), \qquad
	s_5\;=\; s_4 \;+\; \sigma_4 q_2(\omega).
	$$
\end{proposition}
\begin{proof}
	\Cref{SM:prop:arc_EP2}.
\end{proof}

\begin{corollary}
	The total arc length of $\mathbf{r}(t)$ is
	\begin{equation}\label{totarclengthII}
		\begin{array}{l}
			{s(1)}\;=\;{s_5} \;=\;(\sigma_0+\sigma_4) q_2(\omega) \;+\; (\sigma_1+\sigma_3) q_3(\omega) \;+\; \sigma_2 \frac{q_4(\omega)}{q_1(\omega)}\\ \\
			\qquad\;\;= q_2(\omega) \big( |\mathbf{A}_0|^2 + |\mathbf{A}_2|^2 \big) + q_4(\omega) \frac{q_0(\omega)}{q_1(\omega)}   |\mathbf{A}_1|^2 +q_3(\omega) \frac{1}{2}(\mathbf{A}_1 \mathbf{A}_0^* + \mathbf{A}_0 \mathbf{A}_1^*) \\ \\
			\qquad\qquad\qquad+\; q_3(\omega) \frac{1}{2}(\mathbf{A}_1 \mathbf{A}_2^* + \mathbf{A}_2 \mathbf{A}_1^*)  +
			q_4(\omega)  \, \frac{1}{2} \, (\mathbf{A}_2 \mathbf{A}_0^* + \mathbf{A}_0 \mathbf{A}_2^*).
		\end{array}
	\end{equation}
\end{corollary}

\begin{remark}
	When $\omega \rightarrow 0$, the total arc length formula in \eqref{totarclengthII} yields
	\begin{multline*}
		s(1)\;=\;
		\frac{1}{5} \big( |\mathbf{A}_0|^2 + |\mathbf{A}_2|^2 \big) \;+\; \frac{2}{15} |\mathbf{A}_1|^2 \;+\; \frac{1}{10}(\mathbf{A}_1 \mathbf{A}_0^* + \mathbf{A}_0 \mathbf{A}_1^*) \\
		+\;
		\frac{1}{10} (\mathbf{A}_1 \mathbf{A}_2^* + \mathbf{A}_2 \mathbf{A}_1^*)  \;+\;
		\frac{1}{30} (\mathbf{A}_2 \mathbf{A}_0^* + \mathbf{A}_0 \mathbf{A}_2^*),
	\end{multline*}
	thus recovering the well-known result of the quintic polynomial case, see \cite{Fbook}.
\end{remark}


\section{First-order Hermite interpolation by EPH curves}\label{sec:Hermite}

{As in the polynomial case (see, e.g., \cite[Chapter 28.1]{Fbook}), PH curves in ${\cal EP}^{\omega}_1$ could offer the possibility to interpolate $G^1$ Hermite data (i.e., end points and associated unit tangent vectors) at most.
	PH curves in ${\cal EP}^{\omega}_2$ are thus the simplest EPH curves that one could use to match 
	$C^1$ Hermite data.}
The problem of interpolating $C^1$ Hermite data consists in constructing EPH curves that interpolate
prescribed end points $\mathbf{r}_0$, $\mathbf{r}_5$ and first derivatives at these end points, hereinafter denoted by $\mathbf{d}_i$, $\mathbf{d}_f$, respectively. For the sake of conciseness, we also introduce the following abbreviations that do not specify the dependence on $\omega$:
$$
I_0\;:=\;q_2(\omega), \qquad
I_1\;:=\;\frac{1}{2} q_3(\omega), \qquad
I_2\;:=\;\frac{1}{2} q_4(\omega), \qquad
I_3\;:=\; \frac{q_0(\omega)}{q_1(\omega)} \, q_4(\omega).
$$

\begin{proposition}\label{proposizione Pb Hermite EPH 3D}
	The PH curves ${\bf r}(t)$ in $\mathcal{EP}^\omega_2$ solving the first-order Hermite interpolation problem
	$
	\mathbf{r}(0) = \mathbf{r}_0, \; \mathbf{r'}(0) = \mathbf{d}_i, \;
	\mathbf{r}(1) = \mathbf{r}_5, \; \mathbf{r'}(1) = \mathbf{d}_f,
	$
	have control points given by \eqref{eq:cps3dII}
	with
	\begin{equation}\label{eq:A0A1A2_sol}
		\begin{array}{c}
			\mathbf{A}_0\;=\;\sqrt{|\mathbf{d}_i|} \, \frac{{\mathbf{i}}+{\bf w}_i}{|{\mathbf{i}}+{\bf w}_i|} \exp(\eta_0 {\mathbf{i}}), \qquad
			\mathbf{A}_2\;=\;\sqrt{|\mathbf{d}_f|} \, \frac{{\mathbf{i}}+{\bf w}_f}{|{\mathbf{i}}+{\bf w}_f|} \exp(\eta_2 {\mathbf{i}}), \\ \\
			\mathbf{A}_1\;=\;-\frac{I_1}{I_3}(\mathbf{A}_0+\mathbf{A}_2)\;+\;\frac{\sqrt{|{\bf c}|}}{I_3} \, \frac{{\mathbf{i}}+{\bf w}_c}{|{\mathbf{i}}+{\bf w}_c|} \exp(\eta_1 {\mathbf{i}}),
		\end{array}
	\end{equation}
	where
	\begin{equation}\label{eq:c}
		{\bf c}\;:=\;I_3(\mathbf{r}_5-\mathbf{r}_0) \;+\; (I_1^2-I_0 I_3) \, (\mathbf{d}_i+\mathbf{d}_f)\;+\; (I_1^2-I_2 I_3) \, (\mathbf{A}_0\mathbf{i}\mathbf{A}_2^*+\mathbf{A}_2\mathbf{i}\mathbf{A}_0^*),
	\end{equation}
	and
	\begin{itemize}
		\item[$\bullet$] $(\lambda_i,\mu_i,\nu_i),\;(\lambda_f,\mu_f,\nu_f),\;(\lambda_c,\mu_c,\nu_c)$ are the direction cosines of $\mathbf{d}_i$, $\mathbf{d}_f$ and $\mathbf{c}$, respectively;
		\item[$\bullet$] ${\bf w}_i=\lambda_i {\mathbf{i}} + \mu_i {\mathbf{j}} + \nu_i {\mathbf{k}},\;{\bf w}_f=\lambda_f {\mathbf{i}} + \mu_f {\mathbf{j}} + \nu_f {\mathbf{k}},\;
		{\bf w}_c\;=\;\lambda_c {\mathbf{i}} + \mu_c {\mathbf{j}} + \nu_c {\mathbf{k}}$ are unit vectors in the directions of $\mathbf{d}_i$, $\mathbf{d}_f$ and $\mathbf{c}$, respectively;
		\item[$\bullet$] $\eta_0,\;\eta_1,\;\eta_2$ are free angular variables in $\left[-\pi/2,\pi/2\right]$.
	\end{itemize}
\end{proposition}

\begin{proof}
	In view of \eqref{eq:hodograph_quaternionic_formg} and \eqref{eq:At_spatial_quintic},
	interpolation of the end-derivatives yields the equations
	\begin{equation}\label{equaz 2 dim prob Hermite ATPH 3D}
		\mathbf{A}_0\mathbf{i}\mathbf{A}_0^*\;=\;\mathbf{d}_i, \qquad \mathbf{A}_2\mathbf{i}\mathbf{A}_2^*\;=\;\mathbf{d}_f,
	\end{equation}
	for $\mathbf{A}_0$ and $\mathbf{A}_2$, where $\mathbf{d}_i$ and $\mathbf{d}_f$ are known pure vector quaternions.
	Moreover, interpolation of the end points $\mathbf{r}_0$ and $\mathbf{r}_5$ gives the condition
	\begin{equation}\label{equaz 3 dim prob Hermite ATPH 3D}
		\begin{array}{l}
			\int_0^{1}{\mathbf{A}(t)\mathbf{i}\mathbf{A^*}(t) \, dt}\;=\;\mathbf{r}_5-\mathbf{r}_0 \\ \\
			\qquad\qquad\qquad\quad\;\;=\; I_0 \mathbf{A}_0\mathbf{i}\mathbf{A}_0^* + I_1 (\mathbf{A}_0\mathbf{i}\mathbf{A}_1^*+\mathbf{A}_1\mathbf{i}\mathbf{A}_0^*) + I_2 (\mathbf{A}_0\mathbf{i}\mathbf{A}_2^* +\mathbf{A}_2\mathbf{i}\mathbf{A}_0^*)   \\ \\
			\qquad\qquad\qquad\qquad\qquad+\; I_3 \mathbf{A}_1\mathbf{i}\mathbf{A}_1^*+ I_1 (\mathbf{A}_1\mathbf{i}\mathbf{A}_2^*+\mathbf{A}_2\mathbf{i}\mathbf{A}_1^*) + I_0 \mathbf{A}_2\mathbf{i}\mathbf{A}_2^*.
		\end{array}
	\end{equation}
	Recalling the result in \cite[Chapter 28]{Fbook} and \cite[Section 3.2]{Faroukietal}, the quaternion equations \eqref{equaz 2 dim prob Hermite ATPH 3D}
	can be solved directly obtaining
	\begin{equation}\label{equaz 4 dim prob Hermite ATPH 3D}
		\begin{array}{c}
			\mathbf{A}_0 
			\;=\; \sqrt{|\mathbf{d}_i|} \, \frac{{\mathbf{i}}+{\bf w}_i}{|{\mathbf{i}}+{\bf w}_i|} \exp(\eta_0 {\mathbf{i}}) 
			\qquad\text{and}\qquad
			\mathbf{A}_2 
			\;=\; \sqrt{|\mathbf{d}_f|} \, \frac{{\mathbf{i}}+{\bf w}_f}{|{\mathbf{i}}+{\bf w}_f|} \exp(\eta_2 {\mathbf{i}}). 
		\end{array}
	\end{equation} 
	Knowing $\mathbf{A}_0$ and $\mathbf{A}_2$, the solution of \eqref{equaz 3 dim prob Hermite ATPH 3D} for $\mathbf{A}_1$
	may appear more difficult. However, by using \eqref{equaz 2 dim prob Hermite ATPH 3D} and making appropriate rearrangements, \eqref{equaz 3 dim prob Hermite ATPH 3D} can be rewritten as
	\begin{multline}\label{equaz 6 dim prob Hermite ATPH 3D}
		(I_1  \mathbf{A}_0 + I_3 \mathbf{A}_1 + I_1 \mathbf{A}_2)\mathbf{i}(I_1 \mathbf{A}_0 + I_3 \mathbf{A}_1 + I_1 \mathbf{A}_2)^*\;=\\
		\quad=\; I_3(\mathbf{r}_5-\mathbf{r}_0) \;+\; (I_1^2-I_0 I_3) \, (\mathbf{d}_i+\mathbf{d}_f)\;+\; (I_1^2-I_2 I_3) \, (\mathbf{A}_0\mathbf{i}\mathbf{A}_2^*+\mathbf{A}_2\mathbf{i}\mathbf{A}_0^*).
	\end{multline}
	Equation \eqref{equaz 6 dim prob Hermite ATPH 3D} is of the form $\widehat{\mathbf{A}} \mathbf{i} \widehat{\mathbf{A}}^*={\bf c}$
	(exactly as \eqref{equaz 2 dim prob Hermite ATPH 3D})
	where
	\begin{equation}\label{eq:Ahat}
		\widehat{\mathbf{A}}\;:=\;I_1 \mathbf{A}_0 \;+\; I_3 \mathbf{A}_1 \;+\; I_1 \mathbf{A}_2.
	\end{equation} 
	Note that ${\bf c}$ is a known pure vector quaternion.
	Exploiting \eqref{equaz 4 dim prob Hermite ATPH 3D} we can write
	$$
	\mathbf{A}_0\mathbf{i}\mathbf{A}_2^*\;+\;\mathbf{A}_2\mathbf{i}\mathbf{A}_0^*\;=\;
	\sqrt{(1+\lambda_i)|\mathbf{d}_i|(1+\lambda_f)|\mathbf{d}_f|}(a_x\mathbf{i}+a_y\mathbf{j}+a_z\mathbf{k}),
	$$
	where
	$$
	\begin{array}{rcl}
		a_x &=& \cos (\Delta \eta) - \frac{(\mu_i\mu_f+\nu_i\nu_f) \cos (\Delta \eta) +
			(\mu_i\nu_f-\mu_f\nu_i) \sin (\Delta \eta)}{(1+\lambda_i)(1+\lambda_f)},\\ 
		a_y &=& \frac{\mu_i \cos (\Delta \eta) -\nu_i \sin (\Delta \eta)}{1+\lambda_i} + \frac{\mu_f \cos (\Delta \eta) + \nu_f \sin (\Delta \eta)}{1+\lambda_f} ,\\ 
		a_z &=& \frac{\nu_i \cos (\Delta \eta) + \mu_i \sin (\Delta \eta)}{1+\lambda_i} + \frac{\nu_f \cos (\Delta \eta) - \mu_f \sin (\Delta \eta)}{1+\lambda_f},
	\end{array}
	$$
	with $\Delta \eta:=\eta_2-\eta_0$.
	Finally, writing $\mathbf{c}=c_x\mathbf{i}+c_y\mathbf{j}+c_z\mathbf{k}$, the solution of \eqref{equaz 6 dim prob Hermite ATPH 3D}  for $\mathbf{A}_1$ is
	$$
	\begin{array}{rcl}
		\mathbf{A}_1 
		&=& -\frac{I_1}{I_3}(\mathbf{A}_0+\mathbf{A}_2) \;+\; \frac{\sqrt{|{\bf c}|}}{I_3} \, \frac{{\mathbf{i}}+{\bf w}_c}{|{\mathbf{i}}+{\bf w}_c|} \exp(\eta_1 {\mathbf{i}}), 
	\end{array}
	$$
	which concludes the proof.  
\end{proof}

\begin{remark}
	When $\omega\rightarrow0$ the result of \cref{proposizione Pb Hermite EPH 3D} gets back the well-known result of the quintic polynomial case treated in \cite{Faroukietal}.
\end{remark}

\begin{remark} \label{rem:etas}
	{ 
		The three angular variables $\eta_0$, $\eta_1$, $\eta_2$, associated with the quaternions $\mathbf{A}_0$, $\mathbf{A}_1$,
		$\mathbf{A}_2$ respectively, do not identify independent degrees of freedom. Indeed, the control points of spatial EPH Hermite interpolants depend only on $\omega$ and the difference of the angles $\eta_0$, $\eta_1$, $\eta_2$. Thus, without loss of generality, we can assume $\eta_1$ to be fixed and, by introducing the notation $\Delta \eta=\eta_2-\eta_0$ and $\eta_m=(\eta_0+\eta_2)/2$, write $\eta_0=\eta_m- \Delta \eta/2$, $\eta_2=\eta_m+ \Delta \eta/2$. Moreover, while the choice $\eta_k\in\left[-\pi/2,\pi/2\right]$, $k\in\{0,1,2\}$, covers all possible different solutions of the Hermite interpolation problem, these are also recovered by the choice $\eta_k\in[0,\pi]$, $k\in\{0,1,2\}$. Indeed, if we substitute $\mathbf{A}_k$ with $\mathbf{A}_k\mathbf{i}$, $k\in\{0,1,2\}$, in \eqref{eq:cps3dII} we obtain exactly the same control points. However, doing so, \eqref{eq:A0A1A2_sol} has to be multiplied from the right by $\mathbf{i}$, which leads to $\exp((\eta_k+\pi/2)\mathbf{i})$ instead of $\exp(\eta_k\mathbf{i})$, $k\in\{0,1,2\}$.
	}
\end{remark}

{
	\begin{remark}
		As already observed, to recover the result for the planar case with $a_0(t)=a_3(t)=0$, we need to have $\mathbf{A}_0,\mathbf{A}_1,\mathbf{A}_2\;\in\;\text{span}\{\mathbf{i},\mathbf{j}\}$. Therefore we must have $\eta_0,\eta_1,\eta_2\in\{0,\pi\}$ (see Remark \ref{rem:etas}).
		Even if the possible combinations of $\eta_k$, $k \in \{0,1,2\}$ are eight, due to \eqref{eq:cps3dII} we obtain only four different curves. Indeed, if we reason on the signs of $\mathbf{A}_k$, $k \in \{0,1,2\}$, we get the results collected in Table \ref{tab:signs}. Thus, as stated in Remark \ref{rem:etas}, one can obtain the four different planar Hermite  interpolants by fixing $\eta_1=0$ and then choosing $\eta_0,\eta_2\in\{0,\pi\}$. Since fixing $\eta_1=0$ means taking $\mathbf{A}_1$ with the sign $+$, we refer to the four possible planar solutions with the notation $++$, $+-$, $-+$, $--$ to specify the four possible combinations of the signs of $\mathbf{A}_0$ and $\mathbf{A}_2$ that one could consider.
	\end{remark}
	\smallskip
	\begin{table}[h!]
		\caption{All possible sign combinations of $\mathbf{A}_0$, $\mathbf{A}_1$, $\mathbf{A}_2$ and their effects on the associated expressions.}
		\label{tab:signs}
		\begin{center}
			\begin{tabular}{c|c|c||c|c|c}
				$\mathbf{A}_0$ & $\mathbf{A}_1$ & $\mathbf{A}_2$ & $\mathbf{A}_0\mathbf{i}\mathbf{A}_1^*+\mathbf{A}_1\mathbf{i}\mathbf{A}_0^*$ &
				$\mathbf{A}_0\mathbf{i}\mathbf{A}_2^*+\mathbf{A}_2\mathbf{i}\mathbf{A}_0^*$ &
				$\mathbf{A}_1\mathbf{i}\mathbf{A}_2^*+\mathbf{A}_2\mathbf{i}\mathbf{A}_1^*$ \\
				\hline
				\hline
				+& +& +& +& +& +\\
				\hline
				+& +& -& +& -& -\\
				\hline
				+& -& +& -& +& -\\
				\hline
				+& -& -& -& -& +\\
				\hline
				-& +& +& -& -& +\\
				\hline
				-& +& -& -& +& -\\
				\hline
				-& -& +& +& -& -\\
				\hline
				-& -& -& +& +& +\\
			\end{tabular}
		\end{center}
	\end{table}
	\smallskip		
	\cref{fig:ATPH_EPH} shows, in the second row, an application of \cref{proposizione Pb Hermite EPH 3D} for planar Hermite data, while an application for spatial Hermite data is illustrated in \cref{fig:Hermite3d_EPH}.}

\begin{figure}[h!]
	\begin{center} 
		\hspace{-0.75cm}\includegraphics[width=0.38\textwidth]{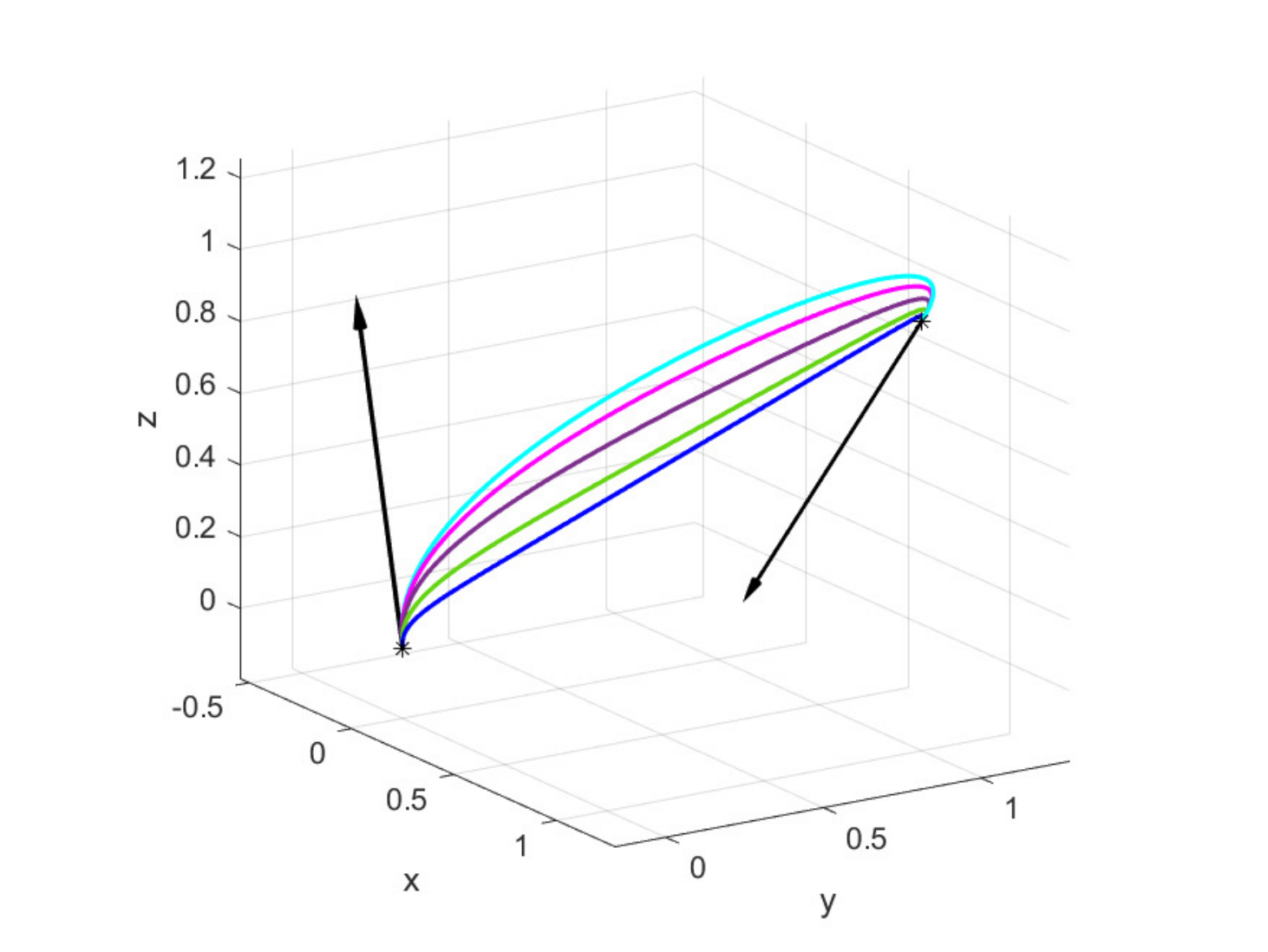}\hspace{-0.9cm}
		\includegraphics[width=0.38\textwidth]{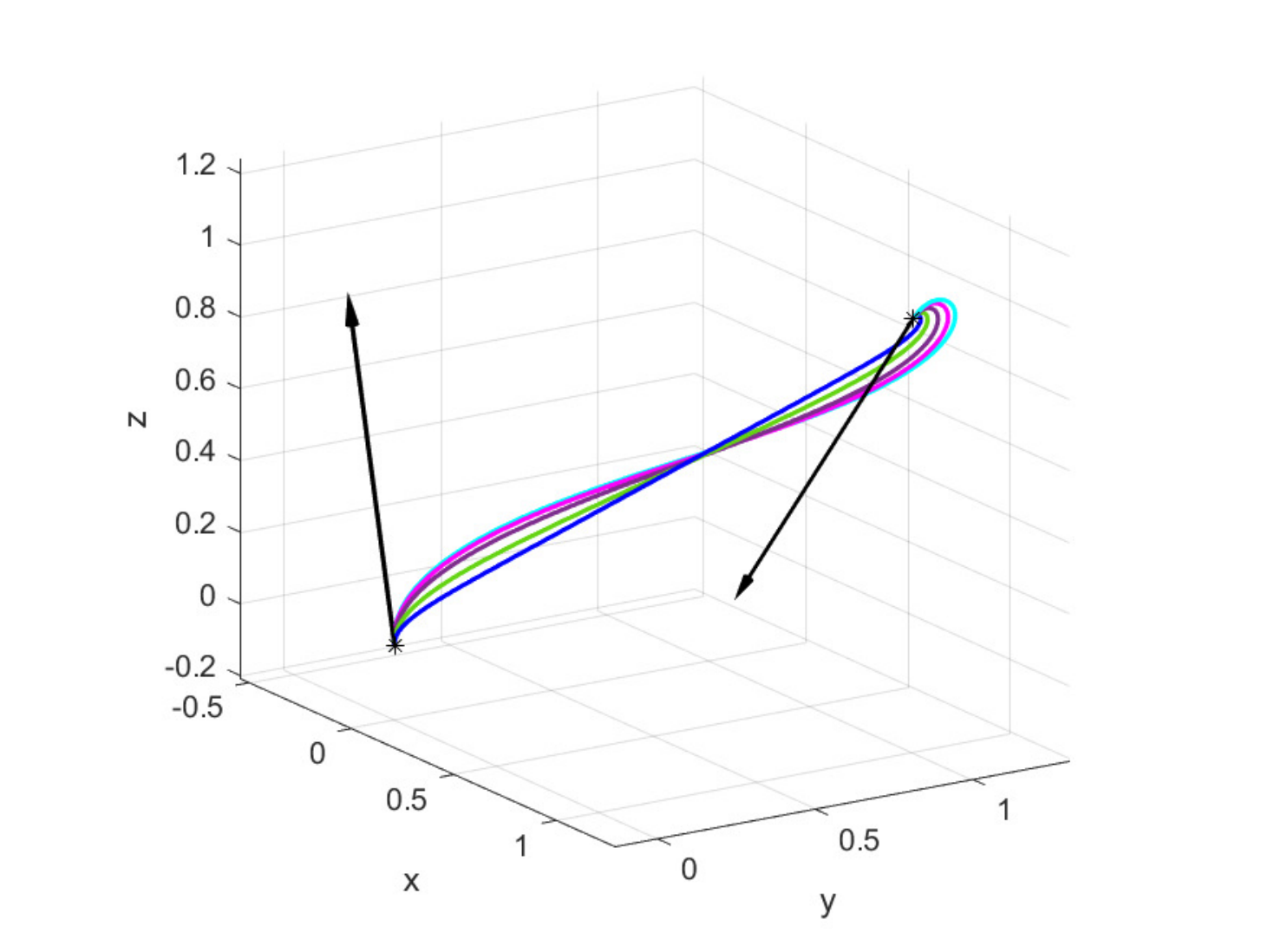}	\hspace{-0.9cm}
		\includegraphics[width=0.38\textwidth]{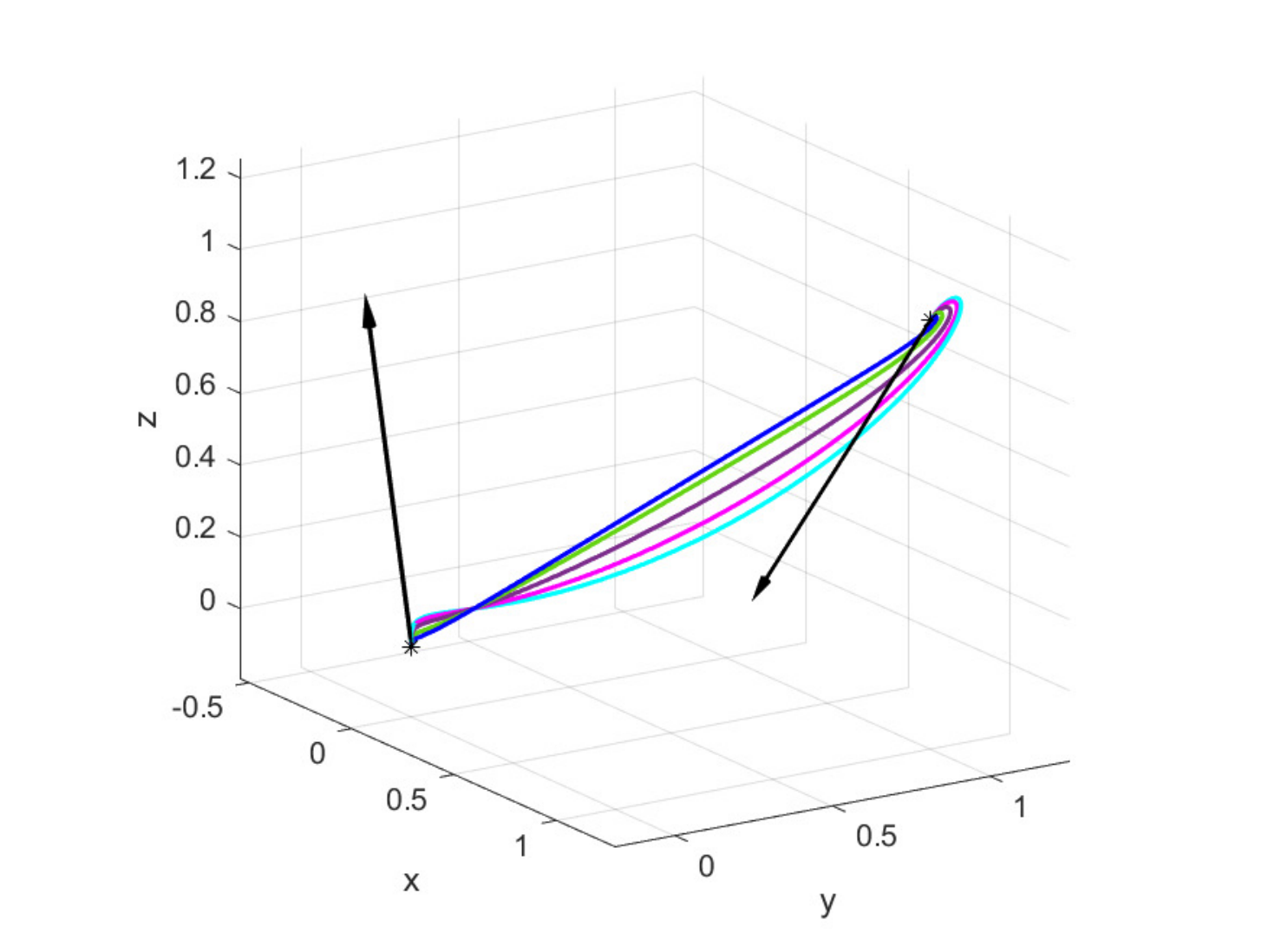}	
	\end{center}
	\caption{
		One-parameter families of spatial EPH Hermite interpolants to the data $\mathbf{r}_0 = (0, 0, 0)$, $\mathbf{r}_5 =
		(1, 1, 1)$, $\mathbf{d}_i = (-0.8, 0.3, 1.2)$, $\mathbf{d}_f = (0.5,-1.3,-1)$, defined by fixing $\eta_1=-\pi/2$, $\Delta \eta=\pi/3$ and $\eta_m=-\pi/2$ (left), $\eta_m=-\pi/10$ (center), $\eta_m=3\pi/10$ (right). The family members of each subfigure are obtained with $\omega \in  \{0.1, 3, 6, 12, 24\}$, where a  bigger value of $\omega$ results in a straighter curve connecting the two endpoints.
	}
	\label{fig:Hermite3d_EPH}
\end{figure}


\section{Evaluation of EPH curves}\label{sec:eval}

In order to evaluate EPH curves two considerations have to be done. On the one hand, looking at the expressions of the normalized B-basis \eqref{eq:U3basis} and \eqref{eq:U5basis}, it is clear that they are not suited for computations when $\omega$ is large. 
The strategy to avoid this problem is to express all the functions involved as a ratio of exponential polynomials, simplifying the dominant growth term. Unfortunately, the resulting expressions are very long. For this reason they are not presented here, but they can be found in \cref{SM:stable_phi}.

On the other hand, computational problems also arise for small values of $\omega$.
Unfortunately, this issue cannot be solved like the previous one with an analytic trick. A way to proceed in this case is to consider for each basis function $\phi^{\omega}_{i,m}(t)$ its corresponding Taylor expansion $T^\omega_{i,m}(t)$ at $\omega=0$ up to a certain order, and then to rely on an efficient algorithm for polynomial evaluation. This is a fair strategy, even from a theoretical point of view, since, for $\omega\rightarrow0$, the considered EPH spaces become exactly polynomial spaces. In our numerical computations we considered $5^{th}$ order Taylor expansions which, for completeness, can be found in \cref{SM:Taylor}.

Here we propose a new ad hoc point-wise evaluation algorithm and we compare it with the de Casteljau-like B-algorithm \cite{MP99} and the recent method proposed by Wo\'zny and Chudy in \cite{WC20}. For the sake of brevity, we only provide a sketch of these two algorithms in \cref{alg:deCast} and \cref{alg:WC}, where the auxiliary functions $\lambda^\omega_{i,j,m}$ and $h^\omega_{j,m}$ are constructed following the strategies detailed in \cite{MP99} and \cite{WC20}, respectively. In order to implement the methods, we recall that all the functions involved must be rewritten in a stable form as the basis functions in \cref{SM:stable_phi}.

Each of these methods has a different running time and a different behavior as $\omega$ approaches $0$. As it is shown in this section, the newly proposed algorithm yields the best results on both fronts and therefore we suggest it as the go-to evaluation algorithm for EPH curves.

\begin{algorithm}[h!]
	\caption{de Casteljau-like} \label{alg:deCast}
	\begin{algorithmic}
		\STATE{Acquire ${\bf r}_0^0:={\bf r}_0,\;{\bf r}_1^0:={\bf r}_1,\; \dots,\; {\bf r}_{2m+1}^0:={\bf r}_{2m+1},\;\hat{t}$}
		\FOR{$k=1,\dots,2m+1$}
		\FOR{$i=0,\dots,2m+1-k$}
		\STATE{${\bf r}_i^{k}\;\longleftarrow\;(1-\lambda^\omega_{i,2m+1-k,m}(\hat{t})) \, {\bf r}_i^{k-1}\;+\;\lambda^\omega_{i,2m+1-k,m}(\hat{t}){\bf r}_{i+1}^{k-1}$}
		\ENDFOR
		\ENDFOR
		\RETURN ${\bf r}_0^{2m+1}$
	\end{algorithmic}	
\end{algorithm}

\begin{algorithm}[h!]
	\caption{Wo\'zny-Chudy} \label{alg:WC}
	\begin{algorithmic}
		\STATE{Acquire $\mathbf{q}_0:={\bf r}_0,\; {\bf r}_1,\;\dots,\;{\bf r}_{2m+1},\;\hat{t}$}
		\FOR{$k=1,\dots,2m+1$}
		\STATE{${\bf q}_k\;\longleftarrow\;(1-h^{\omega}_{2m+1-k,m}(\hat{t})) \, {\bf q}_{k-1}\;+\;h^{\omega}_{2m+1-k,m}(\hat{t}){\bf r}_{k} $}
		\ENDFOR
		\RETURN ${\bf q}_{2m+1}$
	\end{algorithmic}	
\end{algorithm}

We recall that, fixed $d\in\{2,3\}$ and $m\in\{1,2\}$, our interest here is to evaluate the curve 
$
{\bf r}(t)=\sum_{i=0}^{2m+1}  {\bf r}_i  \phi^\omega_{i,m}(t),
$
at a given $\hat t\in[0,1]$, for a set of control points
${\bf r}_i \in \RR^d$, $i=0,\dots,2m+1$. The de Casteljau's algorithm finds the value of $\mathbf{r}(\hat t)$ computing recursively $2m+1$ new sets of points, $\{\mathbf{r}^k_{i}\}_{i=0}^{2m+1-k}$, $k=1,\dots,2m+1$, each having one fewer point than the previous one. At each level $k$, the new set of points is obtained as a convex combination of two consecutive points in the previous level. Instead of computing smaller and smaller sets of control points, Wo\'zny and Chudy's method consists in $2m+1$ convex combinations, each of them adding the contribution of one of the initial control points. A graphical layout of the algorithm can be seen in the first row of \cref{fig:WC_vs_new}. 

\begin{figure}[h!]
	\centering
	\includegraphics[width=0.75\textwidth]{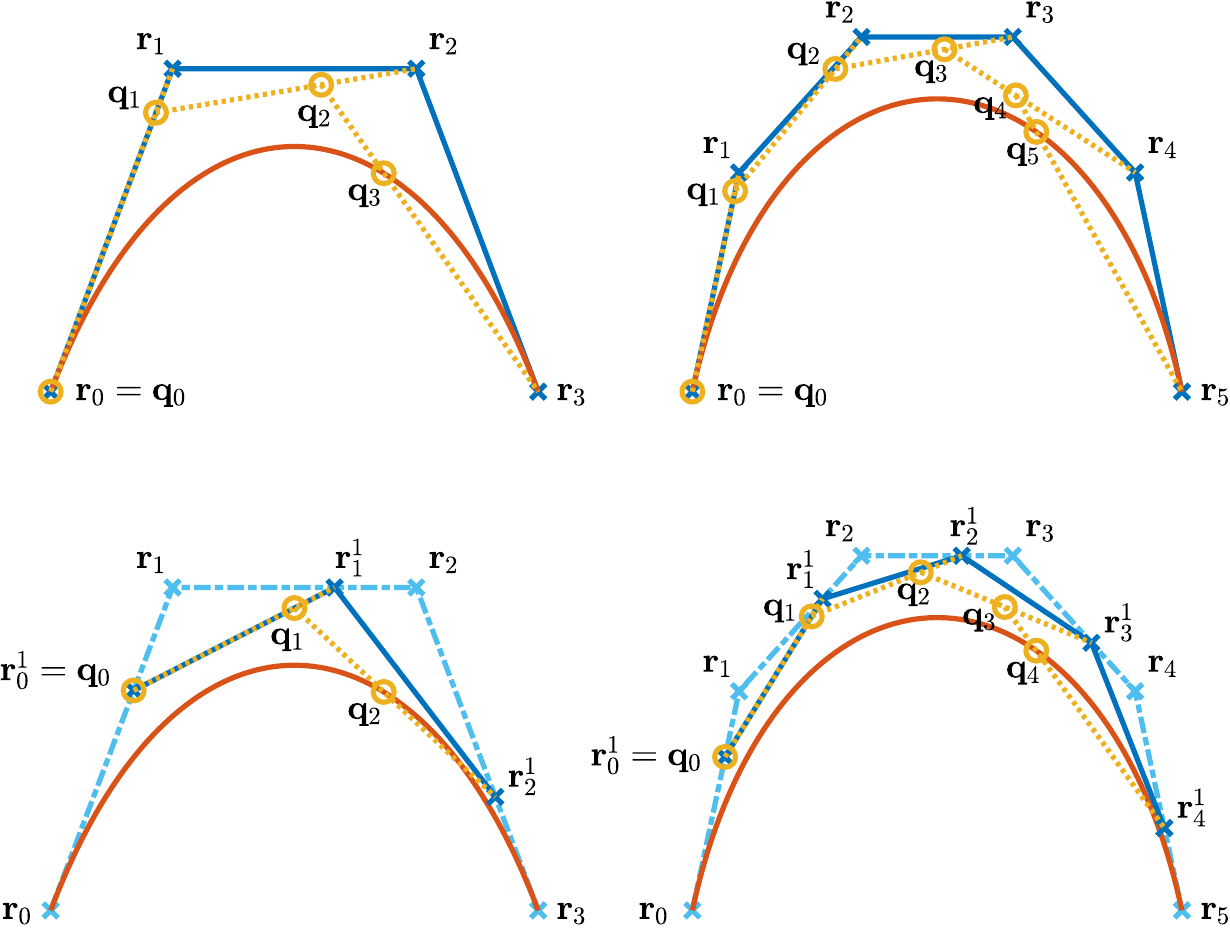}
	\caption{Geometrical comparison between Wo\'zny-Chudy's evaluation algorithm (first row) and the new proposed method (second row) when applied to a curve in ${\cal EP}^\omega_1$ (first column) and a curve in ${\cal EP}^{\omega}_{2}$ (second column).}
	\label{fig:WC_vs_new}
\end{figure}

The new algorithm proposed here fuses in a way both de Casteljau's and Wo\'zny-Chudy's approaches. The idea is to first compute a new set of control vertices, $\{\mathbf{r}^1_i\}_{i=0}^{2m}$, starting from the initial control points, similar to a de Casteljau's step. 
These new vertices are computed such that the associated polynomial B\'ezier curve of degree $2m$ has the same evaluation as ${\bf r}(t)$ at the desired point $\hat t\in(0,1)$, i.e.,
$
{\bf r}(\hat{t})=\sum_{i=0}^{2m}  {\bf r}_i^1  B_{i,2m}(\hat{t})
$
with $B_{i,2m}(\hat{t})=\binom{2m}{i}\, \hat{t}^i \, (1-\hat{t})^{2m-i}$,
where the right-hand side can be efficiently computed via Wo\'zny-Chudy's for polynomial curves, which is much faster than its specialized version for EPH curves. The detailed steps of the method are described in \cref{alg:new}, where, for $m=1$,
\[
\begin{array}{c} 
	\tau^\omega_{0,1}(t)\;=\; {\frac{\phi^\omega_{0,1}(t)}{(1-t)^2}}, \qquad
	
	\tau^\omega_{2,1}(t)\;=\; {1\;-\;\frac{\phi^\omega_{3,1}(t)}{t^2}}, \\ \\
	
	\tau^\omega_{1,1}(t)\;=\;
	{\frac{\phi^\omega_{0,1}(t)\;+\;\phi^\omega_{1,1}(t)\;-\;(1-t)^2}{2t(1-t)}\;=\;
		1\;-\;\frac{\phi^\omega_{2,1}(t)\;+\;\phi^\omega_{3,1}(t)\;-\;t^2}{2t(1-t)}},
\end{array}
\]
and, for $m=2$,
\[
\begin{array}{c} 
	\tau^\omega_{0,2}(t)\;=\; {\frac{\phi^\omega_{0,2}(t)}{(1-t)^4}},\qquad
	
	\tau^\omega_{1,2}(t)\;=\; { \frac{\phi^\omega_{0,2}(t)\;+\;\phi^\omega_{1,2}(t)\;-\;(1-t)^4}{4t(1-t)^3}},\\ \\
	
	\tau^\omega_{2,2}(t)\;=\; { \frac{\sum_{i=0}^2\phi^\omega_{i,2}(t)\;-\;(1-t)^4\;-\;4t(1-t)^3}{6t^2(1-t)^2}
		\;=\; 1\;-\;\frac{\sum_{i=3}^5\phi^\omega_{i,2}(t)\;-\;4t^3(1-t)\;-\;t^4}{6t^2(1-t)^2}},\\ \\
	
	\tau^\omega_{3,2}(t)\;=\; { 1\;-\;\frac{\phi^\omega_{4,2}(t)\;+\;\phi^\omega_{5,2}(t)\;-\;t^4}{4t^3(1-t)}},\qquad 
	
	\tau^\omega_{4,2}(t)\;=\; {1\;-\;\frac{\phi^\omega_{5,2}(t)}{t^4}}.
\end{array}
\]
As for the basis functions, the stable expressions for $\{\tau^\omega_{j,m}(t)\}_{j=0}^{2m}$ exploited in our implementation can be found in \cref{SM:stable_tau}. A graphical layout of the algorithm can be seen in the second row of \cref{fig:WC_vs_new}.
{
	\begin{remark}
		The functions $\tau^\omega_{j,m}(\hat{t})$ have removable discontinuities in $\hat{t}=0$ and $\hat{t}=1$. These are bypassed by the first two \enquote{if}s in Algorithm \ref{alg:new}. In theory, one should be careful to evaluate for $\hat{t}$ close to $0$ or $1$, e.g., approximating each $\tau^\omega_{j,m}(\hat{t})$ with its truncated Taylor expansion. In practice, while using MATLAB, problems occur only for values of $\hat{t}$ which are extremely close to $0$ and $1$. For instance, evaluation of PH curves in ${\cal EP}_2^{\omega}$ for values of $\hat{t}$ close to $0$ starts giving problems at $\hat{t}\approx 10^{-80}$. Since this limitation does not affect its practical use, for the sake of simplicity, Algorithm \ref{alg:new} does not include any modification to handle that situation.
	\end{remark}
}
\begin{algorithm}[h!]
	\caption{New proposal} \label{alg:new}
	\begin{algorithmic}
		\STATE{Acquire ${\bf r}_0,\; \dots,\;{\bf r}_{2m+1},\;\hat{t}$}
		{
			\IF{$\hat{t}=0$}
			\RETURN{$\mathbf{r}_0$}
			\ELSIF{$\hat{t}=1$}
			\RETURN{$\mathbf{r}_{2m+1}$}
			\ELSE
			\FOR{$j=0,\dots,2m$}
			\STATE{${\bf r}_j^1\;\longleftarrow\; \tau^\omega_{j,m}(\hat{t}) {\bf r}_j +  \big(1-\tau^\omega_{j,m}(\hat{t}) \big) {\bf r}_{j+1}$}
			\ENDFOR
			\IF{$\hat{t}\in[0.5,1)$}
			\STATE{$\mathbf{q}_0\;\longleftarrow\;\mathbf{r}_0^1,\quad h_0\;\longleftarrow\;1 \quad\text{and}\quad D\;\longleftarrow\;\frac{1-\hat{t}}{\hat{t}}$}
			\FOR{$k=1,\dots,2m$}
			\STATE{$h_k\;\longleftarrow\;\left( 1+ \frac{k D}{(2m+1-k) h_{k-1}}  \right)^{-1}\quad\text{and}\quad{\bf q}_k\;\longleftarrow\;(1-h_k) \, {\bf q}_{k-1}\;+\;h_k{\bf r}_{k}^1$}
			\ENDFOR
			\ELSIF{$\hat{t}\in(0,0.5)$}
			\STATE{$\mathbf{q}_0\;\longleftarrow\;\mathbf{r}_{2m}^1,\quad h_0\;\longleftarrow\;1 \quad\text{and}\quad D\;\longleftarrow\;\frac{\hat{t}}{1-\hat{t}}$}
			\FOR{$k=1,\dots,2m$}
			\STATE{$h_k\;\longleftarrow\;\left( 1+ \frac{k D}{(2m+1-k) h_{k-1}}  \right)^{-1}\quad\text{and}\quad{\bf q}_k\;\longleftarrow\;(1-h_k) \, {\bf q}_{k-1}\;+\;h_k{\bf r}_{2m-k}^1$}
			\ENDFOR
			\ENDIF
			\RETURN ${\bf q}_{2m}$
			\ENDIF
		}
	\end{algorithmic}	
\end{algorithm}

\subsection{Comparing the three evaluation methods}

We start comparing the behaviour of the three methods as $\omega$ goes to $0$. In order to do so, we computed, for $500$ equispaced values of $\omega\in(0,2]$, the maximum over $100$ curves with random control points uniformly distributed in $(0,1)^d$ of the relative error in the infinity norm committed by each method in approximating the $5^{th}$ order Taylor expansion of the curve at $\omega=0$. In other words, in \cref{fig:Taylor_2}, one can see, for $d=3$ and $m\in\{1,2\}$, the behaviour of the function
\begin{equation} \label{eq:epsilon}
	\rho_{d,m}(\omega)\;:=\;\max_{\{\mathbf{r}_i\}_{i=0}^{2m+1}\in\mathcal{R}}\;\frac{\left\|\;\left.\sum_{i=0}^{2m+1}\mathbf{r}_iT^\omega_{i,m}(t)\right|_{\mathbf{z}}\;-\;
		\left.\sum_{i=0}^{2m+1}\mathbf{r}_i\phi^\omega_{i,m}(t)\right|_{\mathbf{z}}\;\right\|_\infty}{\left\|\;\left.\sum_{i=0}^{2m+1}\mathbf{r}_iT^\omega_{i,m}(t)\right|_{\mathbf{z}}\;\right\|_\infty},
\end{equation}
where $\mathcal{R}$ is a collection of $100$ random sets of $2m+2$ control points in $(0,1)^d$,
$\mathbf{z}=[k/500]_{k=0}^{500}$ and $\sum_{i=0}^{2m+1}\mathbf{r}_i\phi^\omega_{i,m}(t)\big|_{\mathbf{z}}$ is computed with each of the three considered methods. From a theoretical point of view, as $\omega$ gets closer and closer to $0$, an exact evaluation of the curve should approach the evaluation of the polynomial curve obtained substituting each basis function $\phi^\omega_{i,m}$ with its corresponding Taylor polynomial, and thus we should get $\rho_{d,m}(\omega)\rightarrow0$ for $\omega\rightarrow0$. Since stability for small $\omega$ is not achievable, we have that, for each method, the value of $\rho_{d,m}(\omega)$ decreases until a certain threshold is met, under which $\rho_{d,m}(\omega)$ starts to increase and the method becomes unreliable. In particular, from \cref{fig:Taylor_2} it is possible to see how the newly proposed algorithm is the one that can get the closest to $0$ without having numerical issues. For the sake of completeness the points of minimum found for each graph are reported in \cref{tab:confront}. {Therefore, the proposed method is the one that allows exact evaluation for the largest subset of $\omega\in\mathbb{R}^+$.}
\begin{table}[h!]
	\caption{Estimated $\overline{\omega}=\underset{{\omega\in(0,2]}}{\text{argmin}} \,\rho_{d,m}(\omega)$ for $\rho_{d,m}(\omega)$ in \eqref{eq:epsilon} with $d=3$ and $m\in\{1,2\}$ for the three considered methods and the direct evaluation of \eqref{eq:U3basis} and \eqref{eq:U5basis} as a benchmark.} \label{tab:confront} 
	\begin{center}
		\begin{tabular}{c|c|c|c|c}
			\hline
			$m$ & de Casteljau-like & Wo\'zny-Chudy & New proposal & direct evaluation \\
			\hline
			$1$ & 0.2760& 0.2200& 0.0960 & 0.2120 \\
			\hline
			$2$ & 1.1160& 0.3800& 0.1840 & 0.3680 \\
			\hline
		\end{tabular}
	\end{center}
\end{table}

Concerning the running time of the three algorithms, fixed $d=3$ and $m\in\{1,2\}$, for each $\omega\in\{0.0960+2^k\}_{k=-50}^{50}$ we evaluated $10000$ random curves at $501$ equispaced points in $[0,1]$. The results are visible in \cref{fig:time}, where again the new proposed algorithm is the best performing one in both scenarios and for every value of $\omega$ we considered. We observe that the slope around $\omega=10^3$ is due to the fact that most of the exponential functions involved in the computations become very small and thus are set to $0$, speeding up most computations. All numerical experiments were done in MATLAB 2021b on a laptop equipped with an Intel Core i7-10870H CPU and 32 GB RAM.

\begin{figure}
	\centering
	\begin{minipage}{0.49\textwidth}
		\centering
		\includegraphics[width=0.75\textwidth]{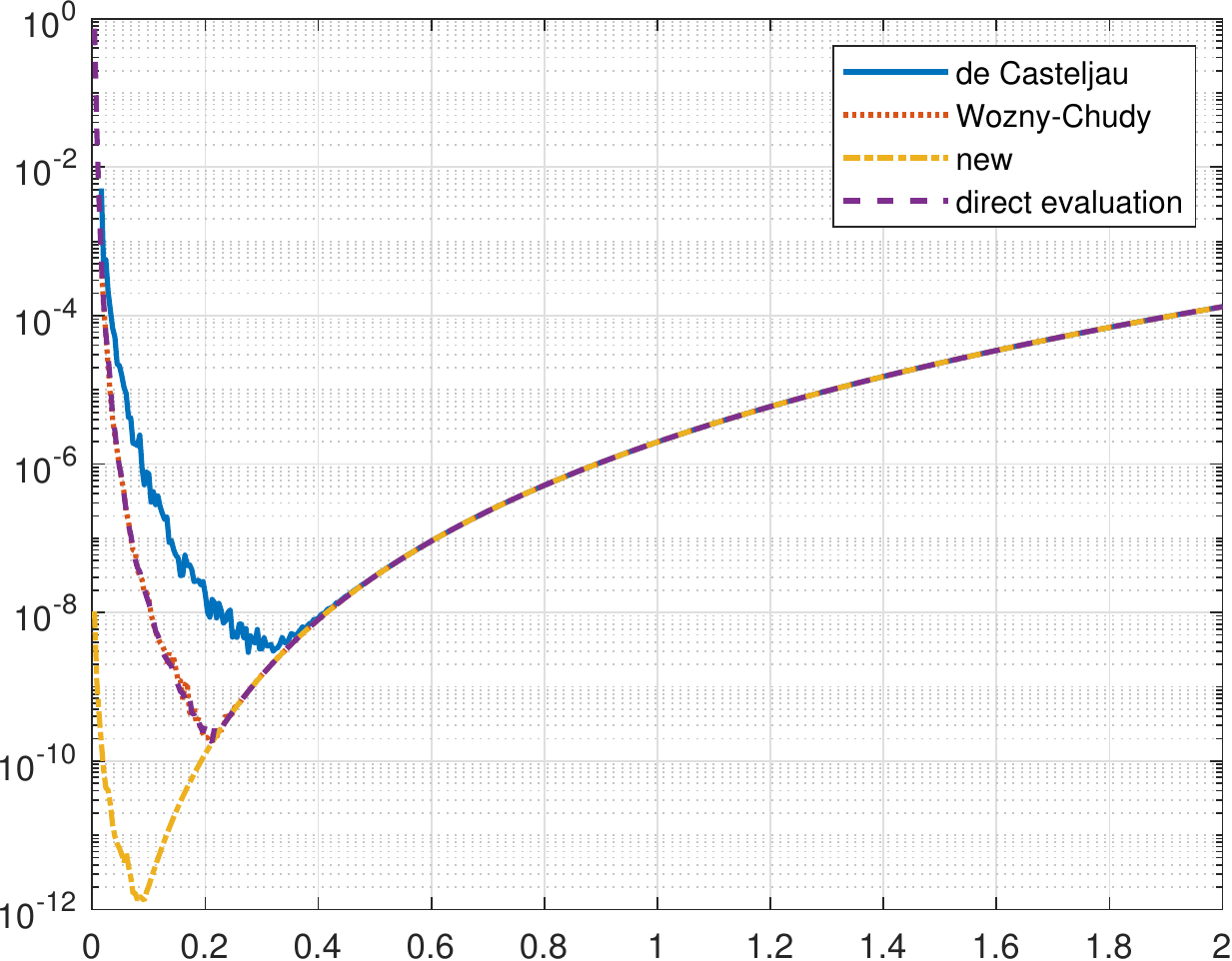}\\
		$m=1$
	\end{minipage}
	\begin{minipage}{0.49\textwidth}
		\centering
		\includegraphics[width=0.75\textwidth]{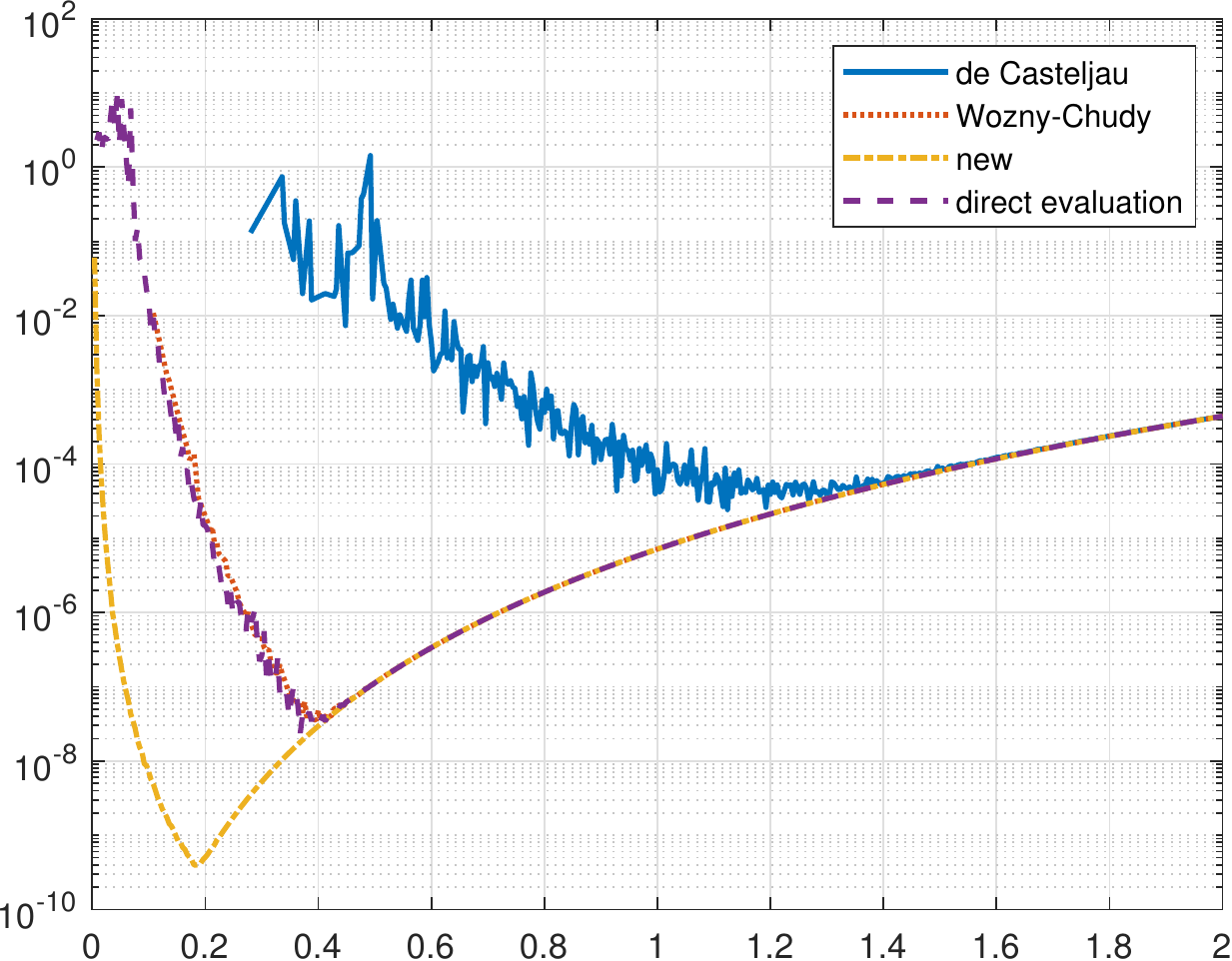}\\
		$m=2$
	\end{minipage}
	\caption{The function $\rho_{d,m}(\omega)$ in \eqref{eq:epsilon} for $d=3$ and $m\in\{1,2\}$ computed with three different methods (the de Casteljau-like B-algorithm, Wo\'zny-Chudy's algorithm and the new proposed method) and using the direct evaluation of \eqref{eq:U3basis} and \eqref{eq:U5basis} as a benchmark.}
	\label{fig:Taylor_2}
\end{figure}

\begin{figure}
	\centering
	\begin{minipage}{0.49\textwidth}
		\centering
		\includegraphics[width=0.75\textwidth]{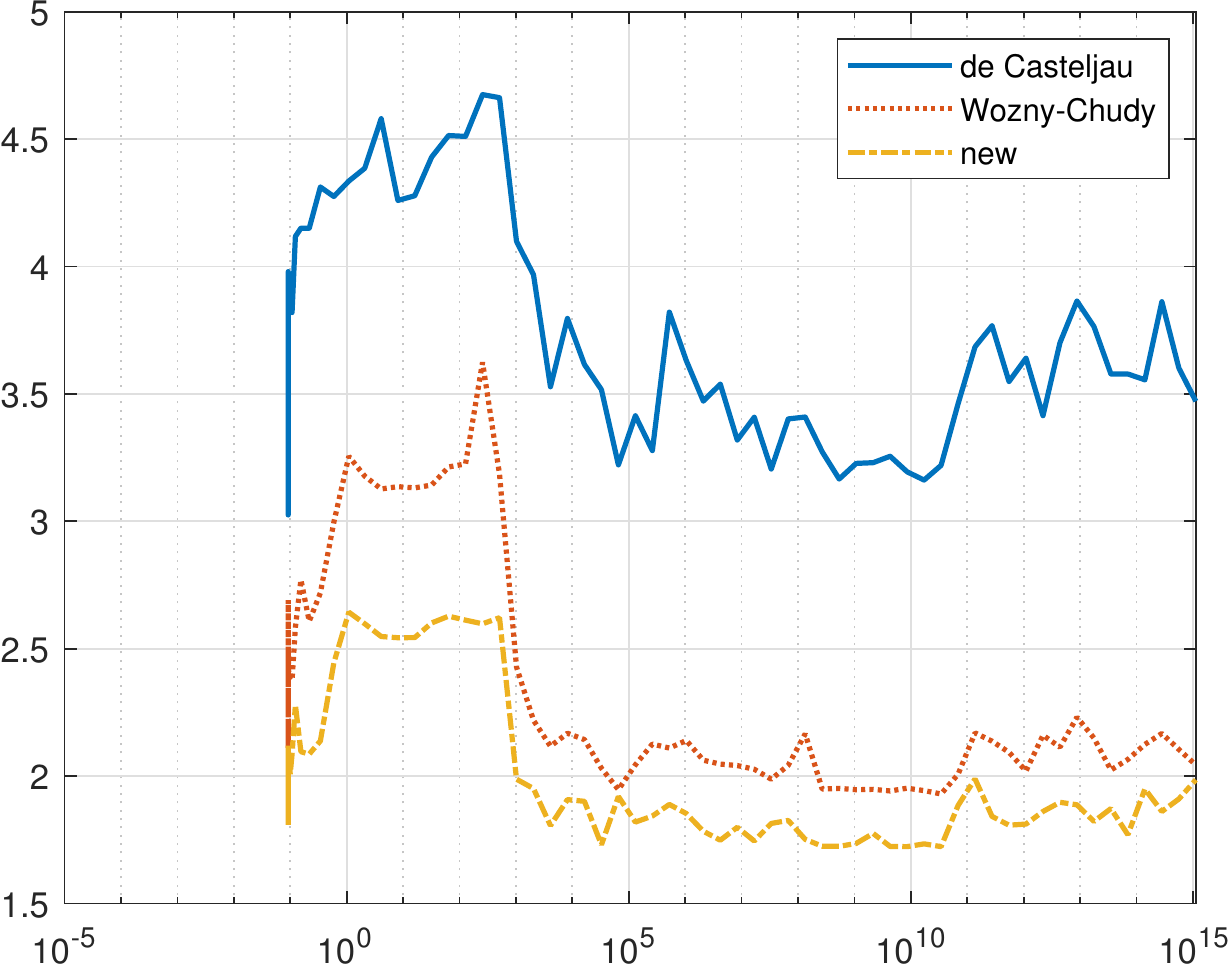}\\
		$m=1$
	\end{minipage}
	\begin{minipage}{0.49\textwidth}
		\centering
		\includegraphics[width=0.75\textwidth]{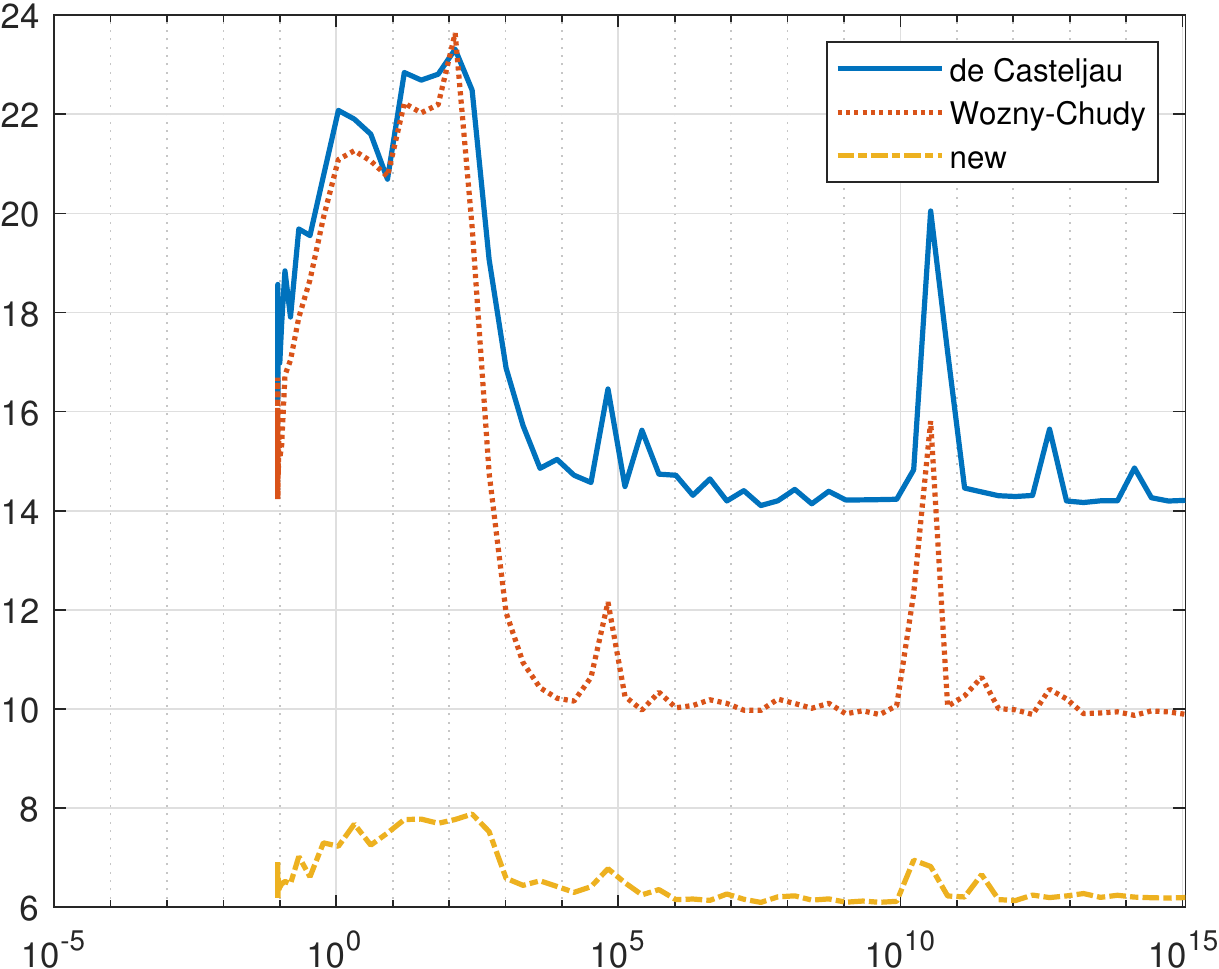}\\
		$m=2$
	\end{minipage}
	\caption{Running times in seconds of the three considered methods for the evaluation of $10000$ random curves, varying $\omega$ in $\{0.0920+2^k\}_{k=-50}^{50}$, for $d=3$ and $m\in\{1,2\}$.}
	\label{fig:time}
\end{figure}

\subsection{A note on a fourth algorithm}
We conclude this section with a short discussion about the dynamic evaluation algorithm presented in \cite{YH17,YH19} which, although can be specialized for EPH curves, presents stability issues for large values of $\omega$. To explain why this is the case, we begin with a brief review of the method. First, it must be that
$\det\left(\left[\mathbf{r}_{2m-d+2},\dots,\mathbf{r}_{2m+1}\right]\right)\neq0$.
Then the method evaluates $\mathbf{r}(t)$ in $k\in\mathbb{N}\setminus\{1\}$ equispaced points over $[0,1]$, finding
$
\mathbf{y}_i=\mathbf{r}(ih)\in\mathbb{R}^{d}
$, $i=0, \ldots, k-1$
where $h=1/(k-1)$. Once the matrices
$\mathbf{R}_1:=\bigl[ \mathbf{r}_0,\dots,\mathbf{r}_{2m-d+1} \bigr]\in\mathbb{R}^{d\times(2m-d+2)}$ and $\mathbf{R}_2:=\bigl[  \mathbf{r}_{2m-d+2},\dots,\mathbf{r}_{2m+1} \bigr]\in\mathbb{R}^{d\times d}$ are defined,
the problem is lifted to dimension $2m+2$, where the new control points are the columns of
$
\mathbf{R}:=\bigl[ \begin{smallmatrix} 
	\mathbf{R}_1,&\mathbf{R}_2\\
	\mathbf{I}_{2m-d+2},&\mathbf{O}_{(2m-d+2)\times d} 
\end{smallmatrix} \bigr]\in\mathbb{R}^{(2m+2)\times(2m+2)},
$
where $\mathbf{I}_{2m-d+2}$ is the $(2m-d+2)$-dimensional identity matrix and $\mathbf{O}_{(2m-d+2)\times d}$ is the  $(2m-d+2)\times d$ matrix of zeros. It is easy to see that $\mathbf{R}$ is invertible with
$
\mathbf{R}^{-1}=\bigl[ \begin{smallmatrix} 
	\mathbf{O}_{(2m-d+2)\times d},&\mathbf{I}_{2m-d+2}\\
	\mathbf{R}_2^{-1},&\mathbf{R}_2^{-1}\;\mathbf{R}_1
\end{smallmatrix} \bigr].
$
Now consider the following recursion:
\begin{equation} \label{eq:dyn_eval_rec}
	\begin{array}{c}
		\mathbf{z}_0\;=\;\mathbf{R}\;\mathbf{e}_1,
		\qquad\mathbf{z}_i\;=\;\mathbf{M}\;\mathbf{z}_{i-1}\;=\;\mathbf{M}^i\;\mathbf{z}_0,\qquad i=1,\dots,k-1,
	\end{array}
\end{equation}
where $\mathbf{e}_1=[\delta_{1,j}]_{j=1}^{2m+2}$, $\delta_{i,j}$ being the Kronecker delta,
$
\mathbf{M}=\mathbf{R}\mathbf{C}^\omega_{h,m}\mathbf{R}^{-1},
$
and $\mathbf{C}^\omega_{h,m}\in\mathbb{R}^{(2m+2)\times(2m+2)}$ is the unique matrix such that
\[
\begin{bmatrix}
	\phi^\omega_{0,m}(t+h)\\
	\vdots\\
	\phi^\omega_{2m+1,m}(t+h)
\end{bmatrix}\;=\;\mathbf{C}^\omega_{h,m}\;\begin{bmatrix}
	\phi^\omega_{0,m}(t)\\
	\vdots\\
	\phi^\omega_{2m+1,m}(t)
\end{bmatrix},\qquad 0\leq t+h\leq1,\; t\in[0,1].
\]
Then
$
\mathbf{y}_i=\begin{bmatrix}
	\mathbf{I}_d,& \mathbf{O}_{d\times(2m+2)}
\end{bmatrix}\mathbf{z}_i
$, $i=0,\dots,k-1$.
In other words, once we have the evaluations of the lifted curve $\{\mathbf{z}_i\}_{i=0}^{k-1}$, we only need to consider the first $d$ components to find the solution for the initial low-dimensional problem.

Let us now focus on the space ${\cal EP}^{\omega}_1$. To use the previous method we need to compute the matrix $\mathbf{C}^\omega_{h,1}$.
Since 
\[
{
	\mathbf{B}}\;\begin{bmatrix}
	\phi^\omega_{0,1}(t)\\
	\phi^\omega_{1,1}(t)\\
	\phi^\omega_{2,1}(t)\\
	\phi^\omega_{3,1}(t)\\
\end{bmatrix}\;=\;\begin{bmatrix}
	1 \\
	t \\
	e^{\omega t}\\
	e^{-\omega t}
\end{bmatrix}\quad\text{ with }\quad
{\mathbf{B}\;=\;
	\begin{bmatrix} 
		1 & 1 & 1 & 1\\ 
		0 & c_2(\omega) & 1-c_2(\omega) & 1\\ 
		1 & 1+\omega  c_2(\omega) & e^\omega(1-\omega  c_2(\omega)) & e^\omega\\ 
		1 & 1-\omega c_2(\omega) & e^{-\omega}(1+\omega  c_2(\omega)) & e^{-\omega}
	\end{bmatrix}	
}
\]
and
\[
\begin{bmatrix}
	1 \\
	t+h \\
	e^{\omega(t+ h)}\\
	e^{-\omega(t+ h)}
\end{bmatrix}\;=\;\widehat{\mathbf{C}}^\omega_{h,1}\;\begin{bmatrix}
	1 \\
	t \\
	e^{\omega t}\\
	e^{-\omega t}
\end{bmatrix},\qquad\text{ with }\qquad\widehat{\mathbf{C}}^\omega_{h,1}\;=\;\begin{bmatrix}
	1 & 0 & 0 & 0\\
	h& 1 & 0 & 0 \\
	0 & 0 & e^{\omega h} & 0 \\
	0 & 0 & 0 & e^{-\omega h}
\end{bmatrix},
\]
we have that
\[
\begin{bmatrix}
	\phi^\omega_{0,1}(t+h)\\
	\phi^\omega_{1,1}(t+h)\\
	\phi^\omega_{2,1}(t+h)\\
	\phi^\omega_{3,1}(t+h)\\
\end{bmatrix}\;=\;\mathbf{C}^\omega_{h,1}\;\begin{bmatrix}
	\phi^\omega_{0,1}(t)\\
	\phi^\omega_{1,1}(t)\\
	\phi^\omega_{2,1}(t)\\
	\phi^\omega_{3,1}(t)\\
\end{bmatrix},\qquad\text{ with }\qquad 
\mathbf{C}^\omega_{h,1}\;=\;
{
	\mathbf{B}^{-1}\;\widehat{\mathbf{C}}^\omega_{h,1}\;\mathbf{B}}.
\]
In particular, it can be shown that
\[\begin{array}{rcl}
	\mathbf{C}^\omega_{h,1}(4,4)&=&\frac{\sinh(\omega(1+h))\;-\;\omega(1+h)}{\sinh(\omega)\;-\;\omega}\;=\;\frac{e^{\omega(1+h)}\;-\;e^{-\omega(1+h)}\;-\;2\omega(1+h)}{e^{\omega}\;-\;e^{-\omega}\;-\;2\omega} \\ \\
	&=&e^{\omega h}\;\frac{1\;-\;e^{-2\omega(1+h)}\;-\;2\omega e^{-\omega(1+h)}}{1\;-\;e^{-2\omega}\;-\;2\omega e^{-\omega}}\;=\;\mathcal{O}(e^{\omega h})\qquad\text{ for }\qquad \omega\rightarrow+\infty,
\end{array}\]
since $h\in(0,1]$. This fact propagates to $\mathbf{M}$ and its powers during the recursion \eqref{eq:dyn_eval_rec} which ends with $\mathbf{M}^{k-1}$ having an element that is $\mathcal{O}(e^\omega)$, making the computations numerically unstable already for $\omega$ of order $10^1$.
In a similar way it is possible to check that the same happens for the space ${\cal EP}^{\omega}_2$. Thus, it is not advisable to use this evaluation method in the context here described.

\appendix

\section{Proof of \cref{prop:EP4}} \label{SM:prop:EP4}
\begin{proof}
	(a) is a consequence of the fact that $\phi_{i,1}^{\omega}(t) \geq 0$ for all $t\in[0,1]$ and $\sum_{i=0}^3 \phi_{i,1}^{\omega}(t)=1$ for all $t \in [0,1]$;\\
	(b) is due to the fact that $\phi_{i,1}^{\omega}(t)=\phi_{3-i,1}^{\omega}(1-t)$ for all $t\in[0,1]$, $i=0,\ldots,3$;\\
	(c) follows from the fact that
	$$
	\begin{array}{l}
		\frac{d}{dt} \phi_{0,1}^{\omega}(t)\;=\; -
		\frac{\varphi_{0,1}^{\omega}(t)}{\int_0^1\varphi^\omega_{0,1}(x)\;dx}, \\ \\
		\frac{d}{dt} \phi_{i,1}^{\omega}(t)\;=\; \frac{\varphi_{i-1,1}^{\omega}(t)}{ \int_0^1\varphi^\omega_{i-1,1}(x)\;dx}\;-\;\frac{\varphi_{i,1}^{\omega}(t)}{ \int_0^1\varphi^\omega_{i,1}(x)\;dx}
		\qquad i\;=\;1,\;2 \\ \\
		\frac{d}{dt} \phi_{3,1}^{\omega}(t)\;=\; \frac{\varphi_{2,1}^{\omega}(t)}{\int_0^1\varphi^\omega_{2,1}(x)\;dx};
	\end{array}
	$$
	(d) is a consequence of (c).
\end{proof}

\section{Proof of \cref{prop:Bez_EP1}} \label{SM:prop:Bez_EP1}
\begin{proof}
	By substituting \eqref{eq:preimage_ep4} into \eqref{eq:hodograph_quaternionic_formg}
	we obtain
	$$
	\begin{array}{lll}
		\mathbf{r}'(t) \hspace{-0.2cm} &=& \hspace{-0.2cm} \mathbf{A}_0 \mathbf{i} \mathbf{A}_0^* \, \left(\psi_{0,1}^{\omega}(t)\right)^2 \,+\,
		(\mathbf{A}_0 \mathbf{i} \mathbf{A}_1^*+\mathbf{A}_1 \mathbf{i} \mathbf{A}_0^*) \, \psi_{0,1}^{\omega}(t) \psi_{1,1}^{\omega}(t) \,+\,
		\mathbf{A}_1 \mathbf{i} \mathbf{A}_1^* \, \left(\psi_{1,1}^{\omega}(t)\right)^2
	\end{array}
	$$
	and thus, in light of \eqref{equivalenze1},
	\begin{equation}\label{eq:hodograph_quaternion}
		\mathbf{r}'(t)=\mathbf{A}_0 \mathbf{i} \mathbf{A}_0^*  \varphi_{0,1}^{\omega}(t) +
		(\mathbf{A}_0 \mathbf{i} \mathbf{A}_1^*+\mathbf{A}_1 \mathbf{i} \mathbf{A}_0^*)  \frac{1}{2} {c_1(\omega)}  \varphi_{1,1}^{\omega}(t) +
		\mathbf{A}_1 \mathbf{i} \mathbf{A}_1^*  \varphi_{2,1}^{\omega}(t).
	\end{equation}
	Since $\mathbf{r}(t)=\mathbf{r}_0+\int_0^t \mathbf{r}'(x) dx$, by integrating the expression in \eqref{eq:hodograph_quaternion} exploiting the formulae in \eqref{integrali1}, 
	we obtain the B\'ezier-like form of $\mathbf{r}(t)$ with control points in \eqref{eq:cps3d}.
\end{proof}

\section{Proof of \cref{prop:par_speed_EP1}} \label{SM:prop:par_speed_EP1}
\begin{proof}
	Since $\sigma(t)=\mathbf{A}(t) \mathbf{A}^*(t)$, in light of \eqref{eq:preimage_ep4} we can write
	$$
	\sigma(t)\;=\;|\mathbf{A}_0|^2 \ \left(\psi_{0,1}^{\omega}(t)\right)^2
	\;+\; (\mathbf{A}_1 \mathbf{A}_0^* + \mathbf{A}_0 \mathbf{A}_1^*) \  \psi_{0,1}^{\omega}(t) \psi_{1,1}^{\omega}(t)
	\;+\; |\mathbf{A}_1|^2 \ \left(\psi_{1,1}^{\omega}(t)\right)^2.
	$$
	Then, exploiting \eqref{equivalenze1}, the claimed result is obtained.
\end{proof}

\section{Proof of \cref{prop:arc_EP1}} \label{SM:prop:arc_EP1}
\begin{proof}
	Since
	$$
	s(t)\;=\;\int_0^t{\sigma(x) dx} \;=\;  \int_0^t \sum_{i=0}^2 \sigma_i \varphi_{i,1}^{\omega}(x) dx \;=\; \sum_{i=0}^2 \sigma_i \int_0^t \varphi_{i,1}^{\omega}(x) dx,
	$$
	then, recalling formulae \eqref{integrali1} we arrive at
	\[
	s(t)\;=\; {\sigma_2 \,  c_2(\omega) \phi_{3,1}^{\omega}(t)\;
		+\; \sigma_1 \, 
		\frac{c_3(\omega)}{c_1(\omega)} 
		\, \sum_{i=2}^3 \phi_{i,1}^{\omega}(t) \;
		+\;\sigma_0 \,  
		c_2(\omega) \, 
		\, \sum_{i=1}^3 \phi_{i,1}^{\omega}(t) \;}
	\]
	and, by collecting the coefficients of each basis function $\phi_{i,1}^{\omega}(t)$, $i=0,\ldots,3$, we get the claimed result.
\end{proof}

\section{Proof of \cref{prop:EP6}} \label{SM:prop:EP6}
\begin{proof}
	(a) is a consequence of the fact that $\phi_{i,2}^{\omega}(t) \geq 0$ for all $t\in[0,1]$ and $\sum_{i=0}^5 \phi_{i,2}^{\omega}(t)=1$ for all $t \in [0,1]$;\\
	(b) is due to the fact that $\phi_{i,2}^{\omega}(t)=\phi_{5-i,2}^{\omega}(1-t)$ for all $t\in[0,1]$, $i=0,\ldots,5$;\\
	(c) follows from the fact that
	$$
	\begin{array}{l}
		\frac{d}{dt} \phi_{0,2}^{\omega}(t)\;=\; -\frac{\varphi_{0,2}^{\omega}(t)}{\int_0^1\varphi^\omega_{0,2}(x)\;dx}, \\ \\
		
		\frac{d}{dt} \phi_{i,2}^{\omega}(t)\;=\;\frac{\varphi_{i-1,2}^{\omega}(t)}{\int_0^1\varphi^\omega_{i-1,2}(x)\;dx}\;-\;\frac{\varphi_{i,2}^{\omega}(t)}{\int_0^1\varphi^\omega_{i,2}(x)\;dx},
		\qquad i=1,\dots,4 \\ \\
		
		\frac{d}{dt} \phi_{5,2}^{\omega}(t)\;=\; \frac{\varphi_{4,2}^{\omega}(t)}{\int_0^1\varphi^\omega_{4,2}(x)\;dx};
	\end{array}
	$$
	(d) is a consequence of (c).
\end{proof}

\section{Proof of \cref{prop:Bez_EP2}} \label{SM:prop:Bez_EP2}
\begin{proof}
	By substituting \eqref{eq:At_spatial_quintic} into \eqref{eq:hodograph_quaternionic_formg} we obtain
	\begin{multline*}
		\mathbf{r}'(t)\;=\;
		\sum_{j=0}^2\;\mathbf{A}_j\mathbf{i}\mathbf{A}_j^* \left(\psi^\omega_{j,2}(t)\right)^2 \;+\; (\mathbf{A}_0 \mathbf{i} \mathbf{A}_1^*+\mathbf{A}_1 \mathbf{i} \mathbf{A}_0^*) \, \psi_{0,2}^{\omega}(t) \psi_{1,2}^{\omega}(t) \\ 
		\qquad+\;  (\mathbf{A}_1 \mathbf{i} \mathbf{A}_2^*+\mathbf{A}_2 \mathbf{i} \mathbf{A}_1^*) \, \psi_{1,2}^{\omega}(t) \psi_{2,2}^{\omega}(t) \;+\; (\mathbf{A}_0 \mathbf{i} \mathbf{A}_2^*+\mathbf{A}_2 \mathbf{i} \mathbf{A}_0^*) \, \psi_{0,2}^{\omega}(t) \psi_{2,2}^{\omega}(t)
	\end{multline*}
	and thus, in light of \eqref{equivalenze2},
	\begin{multline}\label{eq:hodograph_quaternionII}
		\mathbf{r}'(t)\;=\;
		\mathbf{A}_0 \mathbf{i} \mathbf{A}_0^* \, \varphi_{0,2}^{\omega}(t)
		\;+\; (\mathbf{A}_0 \mathbf{i} \mathbf{A}_1^*+\mathbf{A}_1 \mathbf{i} \mathbf{A}_0^*) \, \frac{1}{2} \varphi_{1,2}^{\omega}(t) \\ 
		\qquad\qquad\qquad+\; \left( \mathbf{A}_1 \mathbf{i} \mathbf{A}_1^* \, q_0(\omega)
		+ (\mathbf{A}_0 \mathbf{i} \mathbf{A}_2^*+\mathbf{A}_2 \mathbf{i} \mathbf{A}_0^*) \frac{1}{2} \, q_1(\omega) \, \right) \varphi_{2,2}^{\omega}(t)\\ 
		+\; (\mathbf{A}_1 \mathbf{i} \mathbf{A}_2^*+\mathbf{A}_2 \mathbf{i} \mathbf{A}_1^*) \, \frac{1}{2} \varphi_{3,2}^{\omega}(t)
		\;+\; \mathbf{A}_2 \mathbf{i} \mathbf{A}_2^* \, \varphi_{4,2}^{\omega}(t).
	\end{multline}
	Since $\mathbf{r}(t)=\mathbf{r}_0+\int_0^t \mathbf{r}'(x) dx$, by integrating the expression in \eqref{eq:hodograph_quaternionII} exploiting the formulae in \eqref{integrali2},
	we obtain the B\'ezier-like form of $\mathbf{r}(t)$ with control points in \eqref{eq:cps3dII}. 
\end{proof}

\section{Proof of \cref{prop:par_speed_EP2}} \label{SM:prop:par_speed_EP2}
\begin{proof}
	Since $\sigma(t)=\mathbf{A}(t) \mathbf{A}^*(t)$, in light of \eqref{eq:At_spatial_quintic} we can write
	\begin{multline*}
		\sigma(t)\;=\;\sum_{j=0}^2\;|\mathbf{A}_j|^2 \left(\psi_{j,2}^\omega(t)\right)^2
		\;+\;(\mathbf{A}_1 \mathbf{A}_0^* + \mathbf{A}_0 \mathbf{A}_1^*) \, \psi_{0,2}^{\omega}(t) \psi_{1,2}^{\omega}(t)  \\ 
		+\;(\mathbf{A}_1 \mathbf{A}_2^* + \mathbf{A}_2 \mathbf{A}_1^*) \, \psi_{1,2}^{\omega}(t) \psi_{2,2}^{\omega}(t)
		\;+\;(\mathbf{A}_2 \mathbf{A}_0^* + \mathbf{A}_0 \mathbf{A}_2^*) \, \psi_{0,2}^{\omega}(t) \psi_{2,2}^{\omega}(t).
	\end{multline*}
	Then, exploiting \eqref{equivalenze2}, 
	we obtain the claimed result.
\end{proof}

\section{Proof of \cref{prop:arc_EP2}} \label{SM:prop:arc_EP2}
\begin{proof}
	Since
	$$
	s(t)\;=\;\int_0^t{\sigma(x) dx} \;=\;  \int_0^t \sum_{i=0}^4 \sigma_i \varphi_{i,2}^{\omega}(x) dx \;=\; \sum_{i=0}^4 \sigma_i \int_0^t \varphi_{i,2}^{\omega}(x) dx,
	$$
	recalling formulae \eqref{integrali2}, we then arrive at
	\begin{multline*}
		s(t)\;=\; {
			\sigma_0 \, 
			q_2(\omega) \, \sum_{i=1}^{5} \phi_{i,2}^{\omega}(t)
			\;+\; \sigma_1 \, 
			q_3(\omega) \,\sum_{i=2}^{5} \phi_{i,2}^{\omega}(t)  
			\;+\; \sigma_2 \, 
			\frac{q_4(\omega)}{q_1(\omega)} \, \sum_{i=3}^{5} \phi_{i,2}^{\omega}(t)  } \\
		\;+\;  \sigma_3 \, 
		q_3(\omega) \, \sum_{i=4}^{5} \phi_{i,2}^{\omega}(t) 
		\;+\;  \sigma_4 \,  q_2(\omega) \, \phi_{5,2}^{\omega}(t).
	\end{multline*}
	By collecting the coefficients of each basis function $\phi_{i,2}^{\omega}(t)$, $i=0,\dots,5$, we get the claimed result.
\end{proof}

\section{$5^{th}$ order Taylor expansions at $\omega=0$ of $\{\phi^\omega_{j,m}(t)\}_{j=0}^{2m+1}$, $m\in\{1,2\}$} \label{SM:Taylor}
For $\{\phi^\omega_{j,1}(t)\}_{j=0}^{3}$,
\[\begin{array}{c}
	T^\omega_{0,1}(t)\;=\;T^\omega_{3,1}(1-t),
	\qquad
	T^\omega_{1,1}(t)\;=\;T^\omega_{2,1}(1-t),
	\\ \\
	T^\omega_{2,1}(t)\;=\; 3t^2\left(1-t\right)\;\frac{
		(30t^4-40t^3+23t^2-12t-3)\omega^4
		\;+\;420(3t^2-2t+1)\omega^2
		\;+\;25200
	}{25200},\\ \\
	T^\omega_{3,1}(t)\;=\;t^3\;\frac{
		(10t^4-21t^2+11)\omega^4
		\;+\;420(t^2-1)\omega^2
		\;+\;8400
	}{8400},
\end{array}\]
and, for $\{\phi^\omega_{j,2}(t)\}_{j=0}^{5}$,
\[\begin{array}{c}
	T^\omega_{0,2}(t)\;=\; T^\omega_{5,2}(1-t),
	\qquad
	T^\omega_{1,2}(t)\;=\; T^\omega_{4,2}(1-t),
	\qquad
	T^\omega_{2,2}(t)\;=\; T^\omega_{3,2}(1-t),
	\\ \\
	T^\omega_{3,2}(t)\;=\;
	10t^3{\left(1-t\right)}^2\;\frac{
		(245t^4-392t^3+253t^2-82t+3)\omega^4
		\;+\;420(10t^2-8t+3)\omega^2
		\;+\;35280
	}{35280},\\ \\
	T^\omega_{4,2}(t)\;=\;
	5t^4\left(1-t\right)\; \frac{
		(245t^4-196t^3-96t^2+44t-1)\omega^4
		\;+\;840(5t^2-2t-1)\omega^2
		\;+\;35280
	}{35280},\\ \\
	T^\omega_{5,2}(t)\;=\;
	t^5\;\frac{
		(49t^4-100t^2+51)\omega^4
		\;+\;840(t^2-1)\omega^2
		\;+\;7056
	}{7056}.
\end{array}\]

\section{Stable expressions of $\{\phi^\omega_{j,m}(t)\}_{j=0}^{2m+1}$, $m\in\{1,2\}$, for large $\omega$} \label{SM:stable_phi}
$ $\\ For $m=1$,
\[
\phi^\omega_{0,1}(t)\;=\;\phi^\omega_{3,1}(1-t),
\qquad
\phi^\omega_{1,1}(t)\;=\;\phi^{\omega}_{2,1}(1-t),
\]
\[
\phi^\omega_{2,1}(t)\;=\;\frac{N\phi^\omega_{2,1}(t)}{D\phi^\omega_{2,1}},
\qquad
\phi^\omega_{3,1}(t)\;=\;\frac{e^{-2\omega t}\;+\;2\omega t e^{-\omega t}\;-\;1}{e^{-2\omega }\;+\;2\omega e^{-\omega}\;-\;1}\;e^{\omega(t-1)},
\]
where
\begin{multline*}
	N\phi^\omega_{2,1}(t)\;=\;
	\left(\frac{1}{\omega}+ t\right){e}^{-3\omega}
	\;-\;\frac{1}{\omega}{e}^{\omega\left(t-3\right)}
	\;-\;\left(\frac{1}{\omega}+1\right){e}^{-\omega\left(t+2\right)}
	\;-\;\left(\frac{1}{\omega}+3t-2\right){e}^{-2\omega}\\  \\
	\;+\;\left(\frac{2}{\omega}-1\right){e}^{\omega\left(t-2\right)}
	\;+\;\left(\frac{2}{\omega}+1\right){e}^{-\omega\left(t+1\right)} \\ \\
	\;-\;\left(\frac{1}{\omega}-3t+2\right){e}^{-\omega}	 
	\;-\;\left(\frac{1}{\omega}-1\right){e}^{\omega\left(t-1\right)}
	\;-\;\frac{1}{\omega}{e}^{-t\omega}	
	\;+\;\frac{1}{\omega}- t,
\end{multline*}
\begin{multline*}
	D\phi^\omega_{2,1}\;=\;\left(\frac{2}{\omega}+1\right)e^{-3\omega}\;-\;\left(\frac{2}{\omega}-5-2\omega\right)e^{-2\omega}\;-\;\left(\frac{2}{\omega}+5-2\omega\right)e^{-\omega}\;+\;\frac{2}{\omega}-1.
\end{multline*}
For $m=2$,
\[
\phi^\omega_{0,2}(t)\;=\;\phi^\omega_{5,2}(1-t),
\qquad
\phi^\omega_{1,2}(t)\;=\;\phi^\omega_{4,2}(1-t),
\qquad
\phi^\omega_{2,2}(t)\;=\;\phi^\omega_{3,2}(1-t),
\]
\medskip
\[
\phi^\omega_{3,2}(t)\;=\;\frac{N\phi^\omega_{3,2}(t)}{D\phi^\omega_{3,2}},
\qquad
\phi^\omega_{4,2}(t)\;=\; \frac{N\phi^\omega_{4,2}(t)}{D\phi^\omega_{4,2}}
\;{e}^{\omega\left(t-1\right)},
\]
\medskip
\[
\phi^\omega_{5,2}(t)\;=\;\frac{{e}^{-4\omega t}-8{e}^{-3\omega t}-12\omega t{e}^{-2\omega t}+8{e}^{-\omega t}-1}{{e}^{-4\omega}-8{e}^{-3\omega}-12\omega{e}^{-2\omega}+8{e}^{-\omega}-1}\;{e}^{2\omega\left(t-1\right)},
\]
where
\begin{multline*}
	N\phi^{\omega}_{3,2}(t)\;=\;	
	\left(\frac{3}{\omega}+2t\right){e}^{-5\omega}
	\;-\;\frac{4}{\omega}{e}^{\omega\left(t-5\right)}
	\;+\;\frac{1}{\omega}{e}^{\omega\left(2t-5\right)}
	\;-\;4\left(\frac{2}{\omega}+1\right){e}^{-\omega\left(t+4\right)} \\ \\
	\;+\;\left(\frac{9}{\omega}-2(5t-6)\right){e}^{-4\omega}
	\;-\;4\left(\frac{1}{\omega}+3\right){e}^{\omega\left(t-4\right)}
	\;+\;\left(\frac{3}{\omega}+4\right){e}^{2\omega\left(t-2\right)}\\ \\
	\;+\;\left(\frac{5}{\omega}+2\right){e}^{-\omega\left(2t+3\right)}
	\;+\;4\left(\frac{1}{\omega}-1\right){e}^{-\omega\left(t+3\right)}	
	\;-\;4\left(\frac{3}{\omega}-5t\right){e}^{-3\omega}\\ \\
	+4\left(\frac{3}{\omega}+1\right){e}^{\omega\left(t-3\right)}
	-\left(\frac{9}{\omega}+2\right){e}^{\omega\left(2t-3\right)} 
	-\left(\frac{9}{\omega}-2\right){e}^{-2\omega\left(t+1\right)}
	+4\left(\frac{3}{\omega}-1\right){e}^{-\omega\left(t+2\right)}	\\ \\
	\;-\;4\left(\frac{3}{\omega}+5t\right){e}^{-2\omega}	
	\;+\;4\left(\frac{1}{\omega}+1\right){e}^{\omega\left(t-2\right)} 
	\;+\;\left(\frac{5}{\omega}-2\right){e}^{2\omega\left(t-1\right)} \\ \\
	\;+\;\left(\frac{3}{\omega}-4\right){e}^{-\omega\left(2t+1\right)}	
	\;-\;4\left(\frac{1}{\omega}-3\right){e}^{-\omega\left(t+1\right)}
	\;+\;\left(\frac{9}{\omega}+2(5t-6)\right){e}^{-\omega}\\ \\
	\;-\;4\left(\frac{2}{\omega}-1\right){e}^{\omega\left(t-1\right)}
	\;+\;\frac{1}{\omega}{e}^{-2\omega t}
	\;-\;\frac{4}{\omega}{e}^{-\omega t}
	\;+\;\frac{3}{\omega}-2 t,	
\end{multline*}
\begin{multline*}
	D\phi^{\omega}_{3,2}\;=\;
	2\left[
	\left(\frac{3}{\omega}+1\right){e}^{-5\omega}
	+\left(\frac{27}{\omega}+31+6\omega\right){e}^{-4\omega} 
	-2\left(\frac{15}{\omega}-23-15\omega\right){e}^{-3\omega}\right. \\ \\
	\left.\qquad-2\left(\frac{15}{\omega}+23-15\omega\right){e}^{-2\omega}	
	+\left(\frac{27}{\omega}-31+6\omega\right){e}^{-\omega}
	+\frac{3}{\omega}-1
	\right],
\end{multline*}
\begin{multline*}
	N\phi^{\omega}_{4,2}(t)\;=\; 2{e}^{-\omega\left(2t+5\right)}
	\;+\;3(1+2\omega t){e}^{-\omega\left(t+5\right)}
	\;-\;6{e}^{-5\omega}
	\;+\;{e}^{\omega\left(t-5\right)} 
	\;-\;2{e}^{-\omega\left(3t+4\right)}\\ 
	\;-\;2{e}^{-2\omega\left(t+2\right)}
	\;-\;3(9+10\omega t){e}^{-\omega\left(t+4\right)}
	\;+\;38{e}^{-4\omega} 
	\;-\;7{e}^{\omega\left(t-4\right)}	\\ 
	\;-\;(1+6\omega){e}^{-3\omega\left(t+1\right)}
	\;+\;24(1+\omega){e}^{-\omega\left(2t+3\right)} 
	\;+\;12(2 +\omega(5t-3)){e}^{-\omega\left(t+3\right)}\\ 
	\;-\;8(7-3\omega){e}^{-3\omega}
	\;+\;3(3-2\omega){e}^{\omega\left(t-3\right)} 
	\;+\;3(3+2\omega){e}^{-\omega\left(3t+2\right)}
	\;-\;8(7+3\omega){e}^{-2\omega\left(t+1\right)}\\ 
	\;+\;12(2-\omega(5t-3)){e}^{-\omega\left(t+2\right)} 
	\;+\;24(1-\omega){e}^{-2\omega}	
	\;-\;(1-6\omega){e}^{\omega\left(t-2\right)}\\ 
	\;-\;7{e}^{-\omega\left(3t+1\right)}
	\;+\;38{e}^{-\omega\left(2t+1\right)}
	\;-\;3(9-10\omega t){e}^{-\omega\left(t+1\right)}	
	\;-\;2{e}^{-\omega}\\ 
	\;-\;2{e}^{\omega\left(t-1\right)} 
	\;+\;{e}^{-3\omega t}
	\;-6{e}^{-2\omega t}
	\;+\;3(1-2\omega t){e}^{-\omega t}
	\;+\;2,
\end{multline*}
\begin{multline*}
	D\phi^{\omega}_{4,2}\;=\;
	{e}^{-7\omega}
	\;+\;(1+6\omega){e}^{-6\omega}	
	\;-\;27(3+2\omega){e}^{-5\omega}
	\;+\;(79-156\omega-72\omega^2){e}^{-4\omega} \\ 
	\;+\;(79+156\omega-72\omega^2){e}^{-3\omega}
	\;-\;27(3-2\omega){e}^{-2\omega}
	\;+\;(1-6\omega){e}^{-\omega}
	\;+\;1.
\end{multline*}

\section{Stable expressions of $\{\tau^\omega_{j,m}(t)\}_{j=0}^{2m}$, $m\in\{1,2\}$, for large $\omega$} \label{SM:stable_tau}
$ $\\For $m=1$,
\[
\tau^{\omega}_{0,1}(t)\;=\; 	
\frac{{e}^{\omega \,\left(t-2\right)}-2\,\left(t-1\right)\,\omega \,{e}^{-\omega}-{e}^{-t\,\omega }}{{\left(t-1\right)}^2\,\left({e}^{-2\,\omega }+2\,\omega \,{e}^{-\omega }-1\right)} ,\qquad \tau^\omega_{1,1}(t)\;=\;\frac{N\tau^\omega_{1,1}(t)}{D\tau^\omega_{1,1}(t)},
\]
\[
\tau^{\omega}_{2,1}(t)\;=\; 	
\frac{
	t^2\,{e}^{-2\,\omega }-{e}^{-\omega \,\left(t+1\right)}+2\,t\,\left(t-1\right)\,\omega \,{e}^{-\omega }+{e}^{\omega \,\left(t-1\right)}-t^2
}{
	t^2\,\left({e}^{-2\,\omega }+2\,\omega \,{e}^{-\omega }-1\right)
} ,
\]
where
\begin{multline*}
	N\tau^{\omega}_{1,1}(t)\;=\;
	\left(\;\left(\frac{2}{\omega }+1\right)\,t^2-\left(\frac{4}{\omega }+1\right)\,t+\frac{1}{\omega }\;\right)\,{\mathrm{e}}^{-2\,\omega } 
	\;-\;\frac{1}{\omega }\,{\mathrm{e}}^{\omega \,\left(t-1\right)} \\ \\
	\;-\;\frac{1}{\omega }\,{\mathrm{e}}^{\omega \,\left(t-2\right)} 
	\;+\;2t(t-1)\,{\mathrm{e}}^{-\omega }
	\;+\;\frac{1}{\omega }\,{\mathrm{e}}^{-t\,\omega }\\ \\
	\;+\;\frac{1}{\omega }\,{\mathrm{e}}^{-\omega \,\left(t+1\right)} 
	\;-\;\left(\frac{2}{\omega }-1\right)\,t^2+\left(\frac{4}{\omega }-1\right)\,t-\frac{1}{\omega },
\end{multline*}
\[
D\tau^{\omega}_{1,1}(t)\;=\; 
2\,t\,\left(t-1\right)\,\left({e}^{-\omega}+1\right)\,\left(\left(\frac{2}{\omega}+1\right){e}^{-\omega }-\frac{2}{\omega}+1\right).
\]
For $m=2$, $\tau^\omega_{j,2}(t)=N\tau^\omega_{j,2}(t)/D\tau^\omega_{j,2}(t)$, $j=0,\dots,4$, where
\[
N\tau^{\omega}_{0,2}(t)\;=\; 	
-\;{e}^{2\,\omega \,\left(t-2\right)}
\;+\;8\,{e}^{\omega \,\left(t-3\right)}
\;-\;12\,\left(t-1\right)\,\omega \,{e}^{-2\,\omega }
\;-\;8\,{e}^{-\omega \,\left(t+1\right)}
\;+\;{e}^{-2\,t\,\omega },
\]
\[
D\tau^{\omega}_{0,2}(t)\;=\; 	
{\left(t-1\right)}^4\left(\;
\;-\;{e}^{-4\,\omega }
\;+\;8\,{e}^{-3\,\omega }
\;+\;12\,\omega \,{e}^{-2\,\omega }
\;-\;8\,{e}^{-\omega } 
\;+\;1
\;\right),
\]

\begin{multline*}
	N\tau^{\omega}_{1,2}(t)\;=\; 	
	{\left(t-1\right)}^4\,{e}^{-3\,\omega }
	\;-\;2\,{e}^{\omega \,\left(t-3\right)}
	\;+\;{e}^{\omega \,\left(2\,t-3\right)} \\ 
	\;+\;\,3\left(\;3t^4-12\,t^3+18\,t^2-12\,t+2
	+2t\omega\left(t^3-4\,t^2+6\,t-3\right)\;\right)\,{e}^{-2\,\omega } \\ 
	\;-\;6\,{e}^{\omega \,\left(t-2\right)}
	\;+\;6\,{e}^{-\omega \,\left(t+1\right)} \\ 
	\;-\;3\left(\;3\,t^4-12\,t^3+18\,t^2-12\,t+2
	-2t\omega\left(t^3-4\,t^2+6\,t-3\right)\;\right)\,{e}^{-\omega } \\ 	
	\;-\;{e}^{-2\,t\,\omega }
	\;+\;2\,{e}^{-t\,\omega }
	\;-\;{\left(t-1\right)}^4 
	,
\end{multline*}
\[
D\tau^{\omega}_{1,2}(t)\;=\; 	
4t{\left(t-1\right)}^3\left(\;{e}^{-3\omega }
\;+\;3\left(2\omega +3\right){e}^{-2\omega }
\;+\;3\left(2\omega -3\right) {e}^{-\omega }
\;-\;1
\;\right) ,
\]
\begin{multline*}
	N\tau^{\omega}_{2,2}(t)\;=\; 	
	-\;\left(\;\frac{3(6\,t^4-16\,t^3+12\,t^2-1)}{\omega }\;+\;2t(3\,t^3-8\,t^2+6\,t-1)\;\right){e}^{-2\,\omega }\\ \\
	\;+\;\frac{4}{\omega }{e}^{\omega \,\left(t-2\right)}
	\;-\;\frac{1}{\omega }{e}^{2\,\omega \,\left(t-1\right)}
	\;-\;\frac{4}{\omega }{e}^{-\omega \,\left(t+1\right)} \\ \\
	\;+\;8\,t\,\left(3\,t^3-8\,t^2+6\,t-1\right)\,{e}^{-\omega } \\ \\
	\;+\;\frac{4}{\omega }{e}^{\omega \,\left(t-1\right)}
	\;+\;\frac{1}{\omega }{e}^{-2\,t\,\omega }
	\;-\;\frac{4}{\omega }{e}^{-t\,\omega } \\ \\
	\;-\;\frac{3(6\,t^4-16\,t^3+12\,t^2-1)}{\omega }
	\;+\;2t(3\,t^3-8\,t^2+6\,t-1),
\end{multline*}
\[
D\tau^{\omega}_{2,2}(t)\;=\; 12\,t^2\,{\left(t-1\right)}^2\,	\left(\;\left(\frac{3}{\omega}+1\right){e}^{-2\,\omega }
\;+\;4\,{e}^{-\omega }\;-\;\frac{3}{\omega }\;+\;1\; \right) ,
\]

\begin{multline*}
	N\tau^{\omega}_{3,2}(t)\;=\; 	
	t^3\,\left(3\,t-4\right)\,{e}^{-3\,\omega } 
	\;+\;2\,{e}^{-\omega \,\left(t+2\right)} \\ 
	\;+\;3\left(2\,t\,\omega -8\,t^3\,\omega +6\,t^4\,\omega -12\,t^3+9\,t^4+1\right){e}^{-2\,\omega }\\ 
	\;-\;6\,{e}^{\omega \,\left(t-2\right)} 
	\;+\;{e}^{2\,\omega \,\left(t-1\right)} 
	\;-\;{e}^{-\omega \,\left(2\,t+1\right)} 
	\;+\;6\,{e}^{-\omega \,\left(t+1\right)} \\ 
	\;+\;3\left(2\,t\,\omega -8\,t^3\,\omega +6\,t^4\,\omega +12\,t^3-9\,t^4-1\right){e}^{-\omega } \\ 
	\;-\;2\,{e}^{\omega \,\left(t-1\right)} 
	\;-\;t^3\,\left(3\,t-4\right),
\end{multline*}
\[
D\tau^{\omega}_{3,2}(t)\;=\; 	
4\,t^3\,\left(t-1\right)\,\left(\;{e}^{-3\,\omega }\;+\;3\left(2\,\omega +3\right){e}^{-2\,\omega }\;+\;3\left(2\,\omega -3\right){e}^{-\omega }\;-\;1\;\right) ,
\]

\begin{multline*}
	N\tau^{\omega}_{4,2}(t)\;=\; 	
	-\;t^4\,{e}^{-4\,\omega }
	\;+\;8\,t^4\,{e}^{-3\,\omega }
	\;+\;{e}^{-2\,\omega \,\left(t+1\right)}
	\;-\;8\,{e}^{-\omega \,\left(t+2\right)}\\ 
	\;+\;12\,t\,\left(t^3-1\right)\,\omega \,{e}^{-2\,\omega } 
	\;+\;8\,{e}^{\omega \,\left(t-2\right)}
	\;-\;{e}^{2\,\omega \,\left(t-1\right)}
	\;-\;8\,t^4\,{e}^{-\omega }
	\;+\;t^4,
\end{multline*}
\[
D\tau^{\omega}_{4,2}(t)\;=\; 	
t^4\left(\;-{e}^{-4\,\omega }\;+\;8\,{e}^{-3\,\omega }\;+\;12\,\omega \,{e}^{-2\,\omega }\;-\;8\,{e}^{-\omega }\;+\;1\;\right) .
\]


\section*{Acknowledgments}
This work has been accomplished within the \enquote{Research ITalian network on Approximation} (RITA) and the TAA-UMI group.

\bibliographystyle{siamplain}

\begin{thebibliography}{10}
	
	\bibitem{CP94}
	{\sc J.~M. Carnicer and J.~M. Pe\~{n}a}, {\em Totally positive bases for shape
		preserving curve design and optimality of {$B$}-splines}, Comput. Aided Geom.
	Design, 11 (1994), pp.~633--654,
	\url{https://doi.org/10.1016/0167-8396(94)90056-6}.
	
	\bibitem{CS20}
	{\sc I.~Cattiaux-Huillard and L.~Saini}, {\em Characterization and extensive
		study of cubic and quintic algebraic trigonometric planar {PH} curves}, Adv.
	Comput. Math., 46 (2020), \url{https://doi.org/10.1007/s10444-020-09772-4}.
	
	\bibitem{Fbook}
	{\sc R.~T. Farouki}, {\em Pythagorean-hodograph curves: algebra and geometry
		inseparable}, vol.~1 of Geometry and Computing, Springer, Berlin, 2008,
	\url{https://doi.org/10.1007/978-3-540-73398-0}.
	
	\bibitem{Faroukietal}
	{\sc R.~T. Farouki, M.~al~Kandari, and T.~Sakkalis}, {\em Hermite interpolation
		by rotation-invariant spatial {P}ythagorean-hodograph curves}, Adv. Comput.
	Math., 17 (2002), pp.~369--383,
	\url{https://doi.org/10.1023/A:1016280811626}.
	
	\bibitem{GAPL18}
	{\sc C.~Gonz\'{a}lez, G.~Albrecht, M.~Paluszny, and M.~Lentini}, {\em Design of
		{$C^2$} algebraic-trigonometric pythagorean hodograph splines with shape
		parameters}, Comput. Appl. Math., 37 (2018), pp.~1472--1495,
	\url{https://doi.org/10.1007/s40314-016-0404-y}.
	
	\bibitem{KKRV15}
	{\sc J.~Kozak, M.~Krajnc, M.~Rogina, and V.~Vitrih}, {\em Pythagorean-hodograph
		cycloidal curves}, J. Numer. Math., 23 (2015), pp.~345--360,
	\url{https://doi.org/10.1515/jnma-2015-0023}.
	
	\bibitem{MP99}
	{\sc E.~Mainar and J.~M. Pe\~{n}a}, {\em Corner cutting algorithms associated
		with optimal shape preserving representations}, Comput. Aided Geom. Design,
	16 (1999), pp.~883--906, \url{https://doi.org/10.1016/S0167-8396(99)00035-7}.
	
	\bibitem{MP07}
	{\sc E.~Mainar and J.~M. Pe\~{n}a}, {\em A general class of {B}ernstein-like
		bases}, Comput. Math. Appl., 53 (2007), pp.~1686--1703,
	\url{https://doi.org/10.1016/j.camwa.2006.12.018}.
	
	\bibitem{MP10}
	{\sc E.~Mainar and J.~M. Pe\~{n}a}, {\em Optimal bases for a class of mixed
		spaces and their associated spline spaces}, Comput. Math. Appl., 59 (2010),
	pp.~1509--1523, \url{https://doi.org/10.1016/j.camwa.2009.11.009}.
	
	\bibitem{RM19}
	{\sc L.~Romani and F.~Montagner}, {\em Algebraic-trigonometric
		{P}ythagorean-hodograph space curves}, Adv. Comput. Math., 45 (2019),
	pp.~75--98, \url{https://doi.org/10.1007/s10444-018-9606-8}.
	
	\bibitem{RSA14}
	{\sc L.~Romani, L.~Saini, and G.~Albrecht}, {\em Algebraic-trigonometric
		{P}ythagorean-hodograph curves and their use for {H}ermite interpolation},
	Adv. Comput. Math., 40 (2014), pp.~977--1010,
	\url{https://doi.org/10.1007/s10444-013-9338-8}.
	
	\bibitem{Roth}
	{\sc A.~R\'{o}th}, {\em Algorithm 992: an {O}pen{GL}- and {C}++-based function
		library for curve and surface modeling in a large class of extended
		{C}hebyshev spaces}, ACM Trans. Math. Software, 45 (2019),
	\url{https://doi.org/10.1145/3284979}.
	
	\bibitem{WC20}
	{\sc P.~Wo\'{z}ny and F.~Chudy}, {\em Linear-time geometric algorithm for
		evaluating {B}\'{e}zier curves}, Comput.-Aided Des., 118 (2020), pp.~102760,
	6, \url{https://doi.org/10.1016/j.cad.2019.102760}.
	
	\bibitem{YH17}
	{\sc X.~Yang and J.~Hong}, {\em Dynamic evaluation of free-form curves and
		surfaces}, SIAM J. Sci. Comput., 39 (2017), pp.~B424--B441,
	\url{https://doi.org/10.1137/16M1058911}.
	
	\bibitem{YH19}
	{\sc X.~Yang and J.~Hong}, {\em Dynamic evaluation of exponential polynomial
		curves and surfaces via basis transformation}, SIAM J. Sci. Comput., 41
	(2019), pp.~A3401--A3420, \url{https://doi.org/10.1137/18M1230359}.
	
\end{thebibliography}

\end{document}